\documentclass[10pt]{amsart}

\usepackage{hyperref}

\usepackage{amssymb}
\usepackage{amsmath}
\usepackage{graphicx}
\usepackage{tikz}
\usepackage{amsthm}

\let\newpf\proof \let\proof\relax

\def\DC{{\mathrm{DC}}}

\def\bm{\begin{matrix}}
\def\em{\end{matrix}}

\newcommand{\bt}{\begin{thm}}
\newcommand{\et}{\end{thm}}

\newcommand{\bl}{\begin{lemma}}
\newcommand{\el}{\end{lemma}}

\newcommand{\beq}{\begin{eqnarray}}
\newcommand{\eeq}{\end{eqnarray}}

\newtheorem{thm}{Theorem}[section]

\newtheorem{lemma}[thm]{Lemma}

\theoremstyle{remark}
\newtheorem{rem}{Remark}[section]

\numberwithin{equation}{section}

\def \bn {\hfill \\ \smallskip\noindent}

\theoremstyle{definition}

\newtheorem{definition}{Definition}[section]
\def\proof{\bn {\bf Proof.} }

\def\note#1
{\marginpar

{\tiny $\leftarrow$
\par
\hfuzz=20pt \hbadness=9000 \hyphenpenalty=-100 \exhyphenpenalty=-100
\pretolerance=-1 \tolerance=9999 \doublehyphendemerits=-100000
\finalhyphendemerits=-100000 \baselineskip=6pt
#1}\hfuzz=1pt}

\renewcommand{\mod}{\operatorname{mod}}

\newcommand{\C}{{\mathbb C}}

\newcommand{\N}{{\mathbb N}}
\newcommand{\Q}{{\mathbb Q}}
\newcommand{\R}{{\mathbb R}}
\newcommand{\T}{{\mathbb T}}

\newcommand{\Z}{{\mathbb Z}}

\def\B0{{\bold{0}}}

\catcode`\@=12

\def\Empty{}
\newcommand\oplabel[1]{
  \def\OpArg{#1} \ifx \OpArg\Empty {} \else
  	\label{#1}
  \fi}

\newcommand{\comm}[1]{}
\newcommand{\comment}[1]{}

\begin{document}

\title{Dry Ten Martini problem for the non-self-dual extended Harper's model}
\begin{abstract}
In this paper we prove the dry version of the Ten Martini problem: Cantor spectrum with all gaps open, for the extended Harper's model in the non self-dual region for Diophantine frequencies.
\end{abstract}
\author{Rui Han}
\address{Department of Mathematics, University of California, Irvine, CA, 92697-3875, United States
of America}
\email{rhan2@uci.edu}
\setcounter{tocdepth}{1}

\maketitle

\section{Introduction}

The study of independent electrons on a two-dimensional lattice exposed to a perpendicular magnetic field and periodic potentials can be reduced via an appropriate choice of gauge field to the study of discrete one-dimensional quasiperiodic Jacobi matrices.
The most extensively studied case is the almost Mathieu operator (AMO) acting on $l^2(\Z)$ defined by
\begin{align*}
(H_{\lambda, \alpha, \theta}u)_n=u_{n+1}+u_{n-1}+2\lambda \cos{2\pi (\theta+n\alpha)} u_n.
\end{align*}
This is a one-dimensional tight-binding model with anisotropic nearest neighbor couplings in general.
A more general model, called the extended Harper's model (EHM), is the operator acting on $l^2(\Z)$ defined by:
\begin{align*}
(H_{\lambda, \alpha,\theta}u)_n=c(\theta+n\alpha)u_{n+1}+\tilde{c}(\theta+(n-1)\alpha)u_{n-1}+2\cos{2\pi(\theta+n\alpha)}u_{n}.
\end{align*}
where $c(\theta)=\lambda_1 e^{-2\pi i(\theta+\frac{\alpha}{2})}+\lambda_2+\lambda_3 e^{2\pi i(\theta+\frac{\alpha}{2})}$ and
$\tilde{c}(\theta)=\lambda_1 e^{2\pi i (\theta+\frac{\alpha}{2})}+\lambda_2+\lambda_3 e^{-2\pi i (\theta+\frac{\alpha}{2})}$.
It is obtained when both the nearest neighbor coupling (expressed through $\lambda_2$) and the next-nearest couplings (expressed through $\lambda_1$ and $\lambda_3$) are included.
This model includes AMO as a special case (when $\lambda_1=\lambda_3=0$). 

For the AMO, it was proved in \cite{AJ1} that the spectrum is a Cantor set for any $\alpha\in \R\setminus \Q$ and $\lambda\neq 0$.
This is the {\it Ten Martini Problem} dubbed by Barry Simon, after an offer of Mark Kac.
A much more difficult problem, known as the dry version of the Ten Martini Problem, is to prove that the spectrum is not only a Cantor set, but that all gaps predicted by the Gap-Labelling theorem \cite{BLT}, \cite{JM} are open.
The first result was obtained for Liouvillean $\alpha$ \cite{CEY}, and later it was proved for a set of $(\lambda, \alpha)$ of positive Lebesgue measure \cite{Puig}.
The most recent result is \cite{AJ2}, in which they were able to deal with all Diophantine frequencies and $\lambda\neq 1$.
A solution for all irrational frequencies and $\lambda\neq 1$ was also recently announced in \cite{AYZ}.

Recently, there have been several important advances on the spectral theory of the EHM: purely point spectrum for Diophantine $\alpha$ and a.e.$\theta$ in the positive Lyapunov exponent region \cite{JKS}; the exact formula for Lyapunov exponent for all coupling constants \cite{Jm}; the spectral decomposition for a.e.$\alpha$ \cite{AJM}. However the results that study the spectrum as a set have not been obtained for the EHM.

For EHM, depending on the values of the parameters $\lambda_1, \lambda_2, \lambda_3$, we could divide the parameter space into three regions as shown in the picture below:
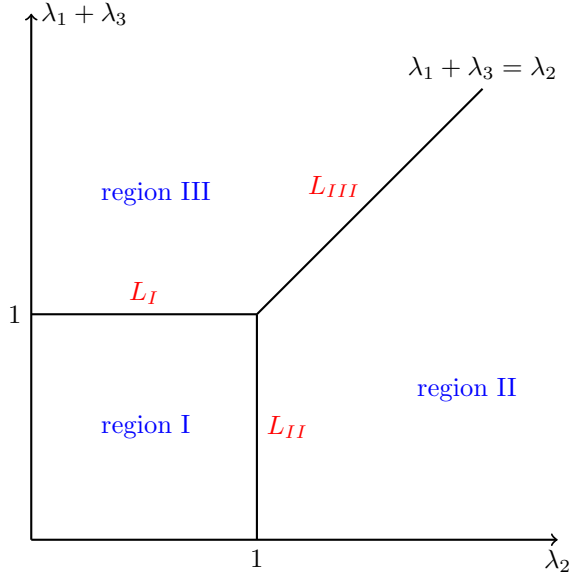
\begin{center}
\begin{tikzpicture}[thick, scale=1]

    \draw[->] (-10,-1) -- (-3,-1) node[below] {$\lambda_2$};
    \draw[->] (-10,-1) -- (-10,6) node[right] {$\lambda_1+\lambda_3$};
    \draw [ ] plot [smooth] coordinates { (- 7, 2) (-4, 5) };
    \draw [ ] plot [smooth] coordinates { (- 7, 2) (-7,-1) };
    \draw [ ] plot [smooth] coordinates { (-10, 2) (-7, 2) };

    \draw(-  4,   5) node [above] {$\lambda_1+\lambda_3=\lambda_2$};
    \draw(-  7,  -1) node [below] {$1$};
    \draw(- 10,   2) node [left]  {$1$};
    \draw(-9.2, 0.5) node [color=blue][right] {region I};
    \draw(-  5,   1) node [color=blue][right] {region II};
    \draw(-9.2, 3.6) node [color=blue][right] {region III};

    \draw(-  7, 0.5) node [color=red][right] {$L_{II}$};
    \draw(-8.5,   2) node [color=red][above] {$L_{I}$};
    \draw(-5.5, 3.7) node [color=red][left] {$L_{III}$};

\end{tikzpicture}
\end{center}

\begin{align*}
&region\ I    : 0<\max{(\lambda_1+\lambda_3,\lambda_2)}<1,\\
&region\ II   : 0<\max{(\lambda_1+\lambda_3, 1)} < \lambda_2,\\
&region\ III  : 0<\max{(1, \lambda_2)} < \lambda_1+\lambda_3.
\end{align*}
According to the action of the duality transformation $\sigma: \lambda=(\lambda_1, \lambda_2, \lambda_3)\rightarrow \hat{\lambda}=(\frac{\lambda_3}{\lambda_2}, \frac{1}{\lambda_2}, \frac{\lambda_1}{\lambda_2})$, region I and region II are dual to each other and region III is a self-dual region.
Region I is the positive Lyapunov exponent region, which is a natural extension of the segment $\{\lambda_1+\lambda_3=0, 0<\lambda_2<1\}$ corresponding to the case $\lambda>1$ in the AMO. 
Region II is the subcritical region, which is an extension of the segment $\{\lambda_1+\lambda_3=0, 1<\lambda_2\}$ corresponding to the case $\lambda<1$ in the AMO.

In this paper we prove the dry version of the Ten Martini Problem in region I and region II under the Diophantine condition. 

Let ${p_n}/{q_n}$ be the continued fraction appoximants of $\alpha\in {\R}\setminus {\Q}$. Let
\begin{align*}
\beta(\alpha)=\limsup_{n\rightarrow\infty} \frac{\ln{q_{n+1}}}{q_n}.
\end{align*}
If $\beta(\alpha)=0$, we say $\alpha$ satisfies the Diophantine condition, denoted by $\alpha\in \mathrm{DC}$. It is easily seen that such $\alpha$ form a full measure subset of $\T$.

It is known that when $E$ is in the closure of a spectral gap, the integrated density of states (IDS) $N(E)\in \alpha \Z+\Z$ (refer to ($\ref{IDS}$) for the definition of IDS) \cite{BLT}, \cite{JM}. Here we prove the inverse is true.
\begin{thm}\label{dry}
If $\alpha\in \mathrm{DC}$ and $\lambda$ belongs to region I or region II, all possible spectral gaps are open.
\end{thm}

\begin{rem}
We note the Dry Ten Martini problem has not yet been solved for the self-dual AMO. In the self-dual region III, Cantor spectrum is known in the isotropic case (when $\lambda_1=\lambda_3$), see Fact $2.1$ in \cite{AJM}. In fact one could prove the operator has zero Lebesgue measure spectrum for all frequencies.
\end{rem}

\begin{rem}
In region I and II, for Liouvillean $\alpha$ (where $\beta(\alpha)$ is large), it is not clear whether even the Cantor spectrum holds. The proof may require a non-trivial adjustment of the proof for AMO in \cite{CEY}.
\end{rem}

We first establish almost localization (see section 3.1) in region I, then a quantitative version of Aubry duality to obtain almost reducibility (see section 3.2) in region II which enables us to deal with all energies whose rotation numbers are $\alpha$-rational.

Thus the strategy follows that of \cite{AJ2}, but we need to extend the almost localization and quantitative duality, as well as the final argument to our Jacobi setting, which is non-trivial on a technical level. At the same time unlike \cite{AJ2}, we only deal with a short-range dual operator, leading to a significant streamlining of some arguments of \cite{AJ2}.

We organize the paper as follows: in section 2 we present some preliminaries, in section 3 we state our main results about almost localization and almost reducibility, relying on which we provide a proof of Theorem $\ref{dry}$. In section 4 and 5 we prove the main results that we present in section 3.

\section{preliminaries}
\subsection{Cocycles}

Let $\alpha\in \R\setminus \Q$ and $A\in C^0(\T, M_2(\C))$ measurable with $\log{\|A(x)\|}\in L^1(\T)$. The quasi-periodic $cocycle$ $(\alpha, A)$ is the dynamical system on $\T\times \C^2$ defined by $(\alpha, A)(x, v)=(x+\alpha, A(x)v)$. The {\it Lyapunov exponent} is defined by
\begin{align*}
L(\alpha, A)=\lim_{n\rightarrow\infty}\frac{1}{n}\int_{\T}\log{\|A_n(x)\|}\mathrm{d}x=\inf_{n}\frac{1}{n}\int_{\T}\log{\|A_n(x)\|}\mathrm{d}x.
\end{align*}
where
\begin{align*}
\begin{cases}
A_n(x)= A(x+(n-1)\alpha)\cdots A(x) \ \ \mathrm{for}\ n\geq 0,\\
A_n(x)=A^{-1}(x+n\alpha)\cdots A^{-1}(x-\alpha) \ \ \mathrm{for}\ n<0.
\end{cases}
\end{align*}

\begin{lemma}\label{uniformupp}(e.g.\cite{AJ2})
Let $(\alpha, A)$ be a continous cocycle, then for any $\delta>0$ there exists $C_{\delta}>0$ such that for any $n\in \N$ and $\theta\in \T$ we have
\begin{align*}
\|A_n(\theta)\|\leq C_{\delta} e^{(L(\alpha, A)+\delta)n}.
\end{align*}
\end{lemma}

We say that $(\alpha, A)$ is {\it uniformly hyperbolic} if there exists continuous splitting $\C^2=E^s(x)\bigoplus E^u(x)$, $x\in \T$ such that for some constant $C, \eta>0$ and all $n\geq 0$, $\|A_n(x) v\|\leq Ce^{-\eta n}\|v\|$ for $v\in E^s(x)$ and $\|A_{-n}(x) v\|\leq C e^{-\eta n}\|v\|$ for $v\in E^{u}(x)$. 

Given two complex cocycles $(\alpha, A^{(1)})$ and $(\alpha, A^{(2)})$, we say they are {\it complex conjugate} to each other if there is  $M\in C^0(\T, SL(2,\C))$ such that
\begin{align*}
M^{-1}(x+\alpha)A^{(1)}(x)M(x)=A^{(2)}(x).
\end{align*}
We assume now that $A$ is a real cocycle, $A\in C^0(\T, SL(2,\R))$.  
The notation of {\it real conjugacy} (between real cocycles) is the same as before, except that we look for $M\in C^0(\T, PSL(2,\R))$. 
A reason why we look for $M\in C^0(\T, PSL(2,\R))$ instead of $M\in C^0(\T, SL(2,\R))$ is given by the following well-known result.
\begin{thm}\label{uhconjugate}
Let $(\alpha, A)$ be uniformly hyperbolic, assume $\alpha\in \mathrm{DC}$ and $A$ analytic, then there exists $M\in C^{\omega}(\T, PSL(2,\R))$ \footnote{In general one cannot take $M\in C^{\omega}(\T, SL(2,\R))$.} such that
$M^{-1}(x+\alpha)A(x)M(x)$ is constant. 
\end{thm}

We say $(\alpha, A)$ is (analytically) {\it reducible} if it is real conjugate to a constant cocycle by an analytic conjugacy.

Let 
\begin{align*}
R_{\theta}=
\left(
\begin{matrix}
\cos{2\pi \theta} \ \ &-\sin{2\pi \theta}\\
\sin{2\pi \theta} \ \ &\cos{2\pi \theta}
\end{matrix}
\right).
\end{align*}
Any $A\in C^0(\T, PSL(2,\R))$ is homotopic to $x\rightarrow R_{\frac{k}{2}x}$ for some $k\in \Z$ called the ${\it degree}$ of $A$, denoted by $\deg{A}=k$.

Assume now that $A\in C^0(\T, SL(2,\R))$ is homotopic to identity.
Then there exists $\phi:\R/\Z \times \R/\Z \to \R$ and $v:\R/\Z \times \R/\Z \rightarrow \R^+$ such that
\begin{align*}
A(x) 
\left(
\begin{matrix}
\cos 2 \pi y \\ 
\sin 2 \pi y 
\end{matrix} 
\right)
=v(x,y)
\left(
\begin{matrix} 
\cos 2 \pi (y+\phi(x,y)) \\
\sin 2 \pi (y+\phi(x,y)) 
\end{matrix} 
\right).
\end{align*}
The function $\phi$ is called a lift of $A$.  
Let $\mu$ be any probability on $\R/\Z \times \R/\Z$ which is invariant under the continuous
map $T:(x,y) \mapsto (x+\alpha,y+\phi(x,y))$, projecting over Lebesgue
measure on the first coordinate.  
Then the number
\begin{align*}
\rho(\alpha,A)=\int \phi\ d\mu \mod \Z
\end{align*}
is independent of the choices of $\phi$ and $\mu$, and is called the
{\it fibered rotation number} of
$(\alpha,A)$.

It can be proved directly by the definition that
\begin{align}\label{rho0}
|\rho(\alpha,A)-\theta|<C\|A-R_{\theta}\|_0.
\end{align}

If $(\alpha,A^{(1)})$ and
$(\alpha,A^{(2)})$ are real conjugate,
$M^{-1}(x+\alpha)A^{(2)}(x)M(x)=A^{(1)}(x)$, and $M:\R/\Z \to PSL(2,\R)$
has degree $k$ then
\begin{align}\label{rhoconju}
\rho(\alpha,A^{(1)})=\rho(\alpha,A^{(2)})-k\alpha/2.
\end{align}

For uniformly hyperbolic cocycles there is the following well-known result.

\begin{thm} \label{uhrho}
Let $(\alpha,A)$ be a uniformly hyperbolic cocycle, with $\alpha \in \R
\setminus \Q$.  Then $2 \rho(\alpha,A) \in \alpha \Z+\Z$.
\end{thm}

\subsection{Extended Harper's model}

We consider the extended Harper's model $\{H_{\lambda,\theta}\}_{\theta \in \T}$.  
The formal solution to $H_{\lambda, \theta}u=Eu$ can be reconstructed via the following equation
\begin{align*}
\left(
\begin{matrix}
u_{n+1}\\
u_n
\end{matrix}
\right)
=
A_{\lambda, E}(\theta+n\alpha)
\left(
\begin{matrix}
u_n\\
u_{n-1}
\end{matrix}
\right).
\end{align*}
where 
$A_{\lambda, E}(\theta)=
\frac{1}{c(\theta)}
\left(
\begin{matrix}
E-2\cos{2\pi \theta}\ \ &-\tilde{c}(\theta-\alpha)\\
c(\theta)\ \ &0
\end{matrix}
\right)$. 
Notice that since $A_{\lambda, E}(\theta)\notin SL(2,\R)$, we introduce the following matrix (see Lemma $\ref{conjugate}$)
\begin{align*}
\tilde{A}_{\lambda, E}(\theta)=
\frac{1}{\sqrt{|c|(\theta)|c|(\theta-\alpha)}}
\left(
\begin{matrix}
E-2\cos{2\pi \theta}\ \ &-|c|(\theta-\alpha)\\
|c|(\theta)\ \ &0
\end{matrix}
\right)=
Q_{\lambda}(\theta+\alpha)A_{\lambda, E}(\theta)Q_{\lambda}^{-1}(\theta),
\end{align*}
where 
$|c|(\theta)=\sqrt{c(\theta)\tilde{c}(\theta)}$ (which is not the same as $|c(\theta)|=\sqrt{c(\theta)\overline{c(\theta)}}$ when $\theta\notin \T$) 
and $Q_{\lambda}(\theta)$ is analytic on $|\mathrm{Im}\theta|\leq \frac{\epsilon_1}{2\pi}$.

The spectrum of $H_{\lambda, \theta}$ denoted by
$\Sigma_{\lambda}$, does not depend on $\theta$ \cite{as},
and it is the set of $E$ such that $(\alpha, \tilde{A}_{\lambda, E})$ is not uniformly
hyperbolic.

The Lyapunov exponent is defined by $L_{\lambda}(E)=L(\alpha,A_{\lambda, E})=L(\alpha, \tilde{A}_{\lambda, E})$. 

For a matrix-valued function $M(\theta)$, 
let $M_{\epsilon}(\theta)=M(\theta+i\epsilon)$ be the phase-complexified matrix. 

In \cite{A3}, Avila divides all the energies in the spectrum into three catagories: super-critical, namely the energy with positive Lyapunov exponent; subcritical, namely the energy whose Lyapunov exponent of the phase-complexified cocycle is identically equal to zero in a neighborhood of $\epsilon=0$; critical, otherwise.

The following theorem is shown in \cite{Jm} (see also the appendix):
\begin{thm}\label{LE} 
Extended Harper's model is {\it super-critical} in region I and {\it sub-critical} in region II. Indeed
\begin{itemize} 
\item when $\lambda$ belongs to region II, $L_{\lambda}(E)=L(\alpha, A_{\lambda, E,\epsilon})=L(\alpha, \tilde{A}_{\lambda, E,\epsilon})=0$ on $|\epsilon|\leq \frac{1}{2\pi}\epsilon_1(\lambda)$, 
\item when $\lambda$ belongs to region II, we have $\hat{\lambda}=(\frac{\lambda_3}{\lambda_2}, \frac{1}{\lambda_2}, \frac{\lambda_1}{\lambda_2})$ belongs to region I and 
\begin{equation}\label{LE1}
L_{\hat{\lambda}}(E)= \epsilon_1(\lambda),
\end{equation}
where
\begin{equation}\label{epsilon1}
\epsilon_1(\lambda)=\ln{\frac{\lambda_2+\sqrt{\lambda_2^2-4\lambda_1\lambda_3}}{\max{(\lambda_1+\lambda_3, 1)}+\sqrt{\max{(\lambda_1+\lambda_3,1)}^2-4\lambda_1\lambda_3}}}>0.
\end{equation}
\end{itemize}
\end{thm}

Fix a $\theta$ and $f\in l^2(\Z)$. Let $\mu_{\lambda, \theta}^f$ be the {\it spectral measure} of $H_{\lambda ,\theta}$ corresponding to $f$,
\begin{align*}
\langle (H_{\lambda, \theta}-z)^{-1}f, f \rangle=\int_{\R}\frac{1}{E-z}\mathrm{d}\mu_{\lambda, \theta}^f (E).
\end{align*}
for $z$ in the resolvent set $\C\setminus \Sigma_{\lambda}$.

The {\it integrated density of states} ($\mathrm{IDS}$) is the function $N_{\lambda}: \R \to [0,1]$ defined by
\begin{align}\label{IDS}
N_{\lambda}(E)=\int_{\T}\mu_{\lambda, \theta}^f(-\infty, E]\mathrm{d}\theta,
\end{align}
where $f\in l^2(\Z)$ is such that $\|f\|_{l^2(\Z)}=1$. It is a continuous non-decreasing surjective funtion.

Notice that $\tilde{A}_{\lambda, E}(\theta)\in SL(2,\R)$ is homotopic to identity in $C^{0}(\T, SL(2,\R))$, in fact just consider
\begin{align*}
H_t(\lambda, E, \theta)=\frac{1}{\sqrt{|c|(\theta)|c|(\theta-t\alpha)}}
\left(
\begin{matrix}
t(E-v(\theta))\ \ &-|c|(\theta-t\alpha)\\
|c|(\theta)   \ \ &0
\end{matrix}
\right).
\end{align*}
which establishes a homotopy of $\tilde{A}_{\lambda, E}(\theta)$ to $R_{\frac{1}{4}}$ and hence to the identity. 
Therefore we can define the rotation number  $\rho(\alpha, \tilde{A}_{\lambda, E})$. Let $\rho_{\lambda}(E)=\rho(\alpha, {\tilde{A}_{\lambda, E}})$.
Notice that $\rho_{\lambda}(E)$ is associated to the operator
\begin{align*}
(\tilde{H}_{\lambda, \theta}u)_n=|c|(\theta+n\alpha)u_{n+1}+|c|(\theta+(n-1)\alpha)u_{n-1}+2\cos{2\pi (\theta+n\alpha)} u_n.
\end{align*}
It is easily seen that for each $\theta$, $\tilde{H}_{\lambda ,\theta}$ and $H_{\lambda ,\theta}$ differ by a unitary operator, thus they share the same spectrum and integrated density of states, $\tilde{N}_{\lambda}(E)=N_{\lambda}(E)$. 
The relation between the integrated density of states and rotation number of $\tilde{H}_{\lambda, \theta}$ yields the following
\begin{equation}\label{IDSROT}
N_{\lambda}(E)=\tilde{N}_{\lambda}(E)=1-2\rho_{\lambda}(E).
\end{equation}

\subsection{The dual model}
It turns out the spectrum $\Sigma_{\lambda}$ of $H_{\lambda, \theta}$ is related to the spectrum $\Sigma_{\hat{\lambda}}$ of $H_{\hat{\lambda}, \theta}$ in the following way
\begin{align*}
\Sigma_{\lambda}=\lambda_2 \Sigma_{\hat{\lambda}}
\end{align*}
by Aubry duality.
This map $\sigma:\lambda\to \hat{\lambda}$ establishes the duality between region I and region II.
The $\mathrm{IDS}$ $N_{\lambda}(E)$ of $H_{\lambda, \theta}$ coincide with the $\mathrm{IDS}$ $N_{\hat{\lambda}}({E}/{\lambda_2})$ of $H_{\hat{\lambda},\theta}$.
Since $\Sigma_{\lambda}=\lambda_2 \Sigma_{\hat{\lambda}}$, we have the following
\begin{thm} \cite{ber}, \cite{sim}
For any $\lambda, \theta$, there exists a dense set of $E\in \Sigma_{\lambda}$ such that there exists a non-zero solution of $H_{\hat{\lambda}, \theta}u=\frac{E}{\lambda_2}u$ with $|u_k|\leq 1+|k|$.
\end{thm}

\subsection{Bounded eigenfunction for every energy}The next result from \cite{AJ2} allows us to pass from a statement of every $\theta$ to every $E$.
\begin{thm}\label{Etheta}\cite{AJ2}
If $E\in \Sigma_{\lambda}$ then there exists $\theta(E) \in \T$ and a bounded solution of $H_{\hat{\lambda},\alpha,\theta}u=\frac{E}{\lambda_2}u$ with $u_0=1$ and $|u_k|\leq 1$.
\end{thm}

\subsection{Localization and reducibility}
\begin{thm}\label{reducible}
Given $\alpha$ irrational, $\theta\in \R$ and $\lambda$ in region II, fix $E\in \Sigma_{\lambda}$, and suppose $H_{\hat{\lambda}, \theta}u=\frac{E}{\lambda_2}u$ has a non-zero exponentially decaying eigenfunction $u={\{u_k\}}_{k\in\Z}$, $|u_k|\leq e^{-c |k|}$ for $k$ large enough. Then the following hold:
\begin{itemize}
\item (A) If $2\theta\notin \alpha\Z+\Z$, then there exists $M: {\R}/{\Z}\rightarrow\ SL(2, \R)$ analytic, such that $${M^{-1}(x+\alpha)}\tilde{A}_{\lambda, E}(x)M(x)=R_{\pm\theta}.$$ 
In this case $\rho (\alpha, \tilde{A}_{\lambda, E})=\pm\theta+\frac{m}{2}\alpha$ $\mathrm{mod} \Z$, where $m=\deg{M}$ (here since $M\in SL(2,\R)$, we have that $m$ is an even number) and $2\rho (\alpha, \tilde{A}_{\lambda,E})\notin \alpha\Z+\Z$.
\item (B) If $2\theta \in \alpha\Z+\Z$ and $\alpha\in\DC$, then there exists $M: {\R}/{\Z}\rightarrow\ PSL(2, \R)$ analytic, such that 
$${M^{-1}(x+\alpha)}\tilde{A}_{\lambda, E}(x)M(x)=\left(\begin{matrix}\pm 1       &a\\ 0      &\pm 1\end{matrix}\right)$$
 with $a\neq 0$. In this case $\rho(\alpha, \tilde{A}_{\lambda, E})=\frac{m}{2}\alpha$ $\mathrm{mod}\Z$, where $m=\deg{M}$, i.e. $2\rho(\alpha, \tilde{A}_{\lambda, E})\in \alpha\Z+\Z$.
\end{itemize}
\end{thm}

\begin{proof}
Let $u(x)=\sum_{k\in\Z} \hat{u}_k e^{2\pi i k x}$, $U(x)=\left(\begin{matrix} e^{2\pi i \theta}u(x)\\ u(x-\alpha)\end{matrix}\right)$. Then
\begin{align*}
A_{\lambda , E}(x)U(x)=e^{2\pi i \theta}U(x+\alpha),
\end{align*}
\begin{align*}
\tilde{A}_{\lambda , E}(x)\tilde{U}(x)=e^{2\pi i \theta}\tilde{U}(x+\alpha).
\end{align*}
Notice $\tilde{U}(x)=Q_{\lambda}(x)U(x)$ is analytic in $|\mathrm{Im}x|<\frac{\tilde{c}}{2\pi}$, where $\tilde{c}=\min{(\epsilon_1,c)}$, $\epsilon_1$ as in $\ref{epsilon1}$ and $Q_{\lambda}$ as in $\ref{conjugate}$. Define $\overline{\tilde{U}(x)}$ to be the complex conjugate of $\tilde{U}(x)$ on $\T$ and its analytic extension to $|\mathrm{Im}x|<\frac{\tilde{c}}{2\pi}$. 
Let $M(x)$ be the matrix with columns $\tilde{U}(x)$ and $\overline{\tilde{U}(x)}$. Then,
\begin{align*}
\tilde{A}_{\lambda, E}(x)M(x)=M(x+\alpha)\left(\begin{matrix}e^{2\pi i\theta}    &0\\ 0   &e^{-2\pi i\theta}\end{matrix}\right)\ \ \mathrm{on}\ \T.
\end{align*}
Then since $\det{M(x+\alpha)}=\det{M(x)}$, we know $\det{M(x)}$ is a constant on $\T$.

Case 1. If $\det{M(x)}\neq 0$, then let $M(x)=\tilde{M}(x)\left(\begin{matrix}1    &1\\ i  &-i\end{matrix}\right)$.

\begin{align*}
{\tilde{M}^{-1}(x+\alpha)}\tilde{A}_{\lambda, E}(x)\tilde{M}(x)=R_{\theta}=
\left(
\begin{matrix}
\cos{2\pi \theta}\ \ &-\sin{2\pi \theta}\\
\sin{2\pi \theta}\ \ &\cos{2\pi \theta}
\end{matrix}
\right).
\end{align*}

Case 2. If $\det{M(x)}=0$, then if we denote $\tilde{U}(x)=\left(\begin{matrix}u_1(x) \\ u_2(x)\end{matrix}\right)$, 
then $\det{M(x)}=0$ means there exists $\eta(x)$ such that $u_1(x)=\eta (x)\overline{u_1(x)}$ and $u_2(x)=\eta (x)\overline{u_2(x)}$. 
This implies that $\eta(x)\in \C^{\omega}(\T, \C)$, and $|\eta(x)|=1$ on $\T$. 
Therefore there exists $\phi(x)\in \C^{\omega}(\R/{2\Z}, \C)$ such that $\phi^2(x)=\eta(x)$ and $|\phi(x)|=1$. 
It is easy to see $\overline{\phi(x)}u_1(x)=\phi(x)\overline{u_1(x)}$ and $\overline{\phi(x)}u_2(x)=\phi(x)\overline{u_2(x)}$. 
Then we define $W(x)=\left(\begin{matrix}\overline{\phi(x)}u_1(x)\\ \overline{\phi(x)}u_2(x)\end{matrix}\right)$, it is a real vector on ${\R}/{2\Z}$ with $W(x+1)=\pm W(x)$, and $\tilde{U}(x)=\phi(x) W(x)$. 
Now let us define $\tilde{M}(x)$ to be the matrix with columns $W(x)$ and $\frac{1}{{\|W(x)\|}^{-2}}R_{\frac{1}{4}}W(x)$, 
then $\det{\tilde{M}(x)}=1$ and $\tilde{M}(x)\in PSL(2, \R)$. 
Since 
$$\tilde{A}_{\lambda, E}(x)W(x)=\frac{e^{2\pi i\theta}\phi(x+\alpha)}{\phi(x)}W(x+\alpha).$$ We have
\begin{align*}
\tilde{A}_{\lambda, E}(x)\tilde{M}(x)=\tilde{M}(x+\alpha)\left(\begin{matrix} d(x)   &\tau(x)\\ 0  &{d(x)}^{-1}\end{matrix}\right)
\end{align*}
where $d(x)=\frac{e^{2\pi i\theta}\phi(x+\alpha)}{\phi(x)}$, $|d(x)|=1$ and $d(x)$ being real number, therefore $d(x)=\pm 1$. 
Also $\tau(x)\in \C^{\omega}({\R}/{2\Z}, \C)$. 
But in fact ${\tilde{M}^{-1}(x+\alpha)}\tilde{A}_{\lambda, E}(x)\tilde{M}(x)$ is well-defined on $\T$. Therefore $\tau(x)\in \C^{\omega}(\T, \C)$. Now since we assumed $\alpha\in\DC$, we can further reduce $\tau(x)$ to the constant $\tau=\int_{\T}\tau(x) \mathrm{d}x$. In fact there exists $\psi(x)\in \C^{\omega}(\T, \C)$ such that $-\psi(x+\alpha)+\psi(x)+\tau(x)=\int_{\T}\tau(x) \mathrm{d}x$. This implies

\begin{align*}
\left(\begin{matrix} 1   &-\psi(x+\alpha)\\ 0  &1\end{matrix}\right)\tilde{M}^{-1}(x+\alpha)\tilde{A}_{\lambda, E}(x)\tilde{M}(x)\left(\begin{matrix} 1   &\psi(x)\\ 0   &1\end{matrix}\right)=\left(\begin{matrix} \pm1   &\tau\\ 0   &\pm1\end{matrix}\right).
\end{align*}

In fact if $\det{M(x)}=0$, then $\frac{e^{2\pi i\theta}\phi(x+\alpha)}{\phi(x)}=\pm 1$, which implies that $2\theta\in \alpha\Z+\Z$. 
Therefore if $2\theta\notin \alpha\Z+\Z$, we must be in case (A).
If on the other hand, $2\theta\in \alpha\Z+\Z$, $2\theta=k\alpha+n$, suppose $\tilde{M}^{-1}(x+\alpha)\tilde{A}_{\lambda, E}(x)\tilde{M}(x)=R_{\theta}$, then $R_{-\frac{k}{2}(x+\alpha)}\tilde{M}^{-1}(x+\alpha)\tilde{A}_{\lambda, E}(x)\tilde{M}(x)R_{\frac{k}{2}x}=R_{\frac{n}{2}}=\pm I$ leading to a contradiction. 
Therefore if $2\theta\in \alpha\Z+\Z$, we must be in case (B). $\hfill{} \Box$
\end{proof}

\subsection{Continued fractions}
Let $\{q_n\}$ be the denominators of the continued fraction approximants of $\alpha$. We recall the following properties:
\begin{align*}
\|q_n\alpha\|_{\R/\Z}=\inf_{1\leq |k|\leq q_{n+1}-1} \|k\alpha\|_{\R/\Z},
\end{align*}
\begin{align*}
\frac{1}{2q_{n+1}}\leq \|q_n\alpha\|_{\R/\Z}\leq \frac{1}{q_{n+1}}.
\end{align*}

Recall that the Diophantine condition of $\alpha$ is $\beta(\alpha)=\limsup_{n\rightarrow\infty} \frac{\ln{q_{n+1}}}{q_n}=0$.
Thus for any $\xi>0$, there exists $C_{\xi}>0$ such that 
\begin{equation}
\|k\alpha\|_{\R/\Z}\geq C_{\xi}e^{-\xi |k|}\ \ \mathrm{for}\ \mathrm{any}\ k\neq 0.
\end{equation}

\begin{lemma}\label{smallest}\cite{AJ1}
Let $\alpha\in \R\backslash\Q$, $x\in\R$ and $0\leq l_0\leq q_n-1$ be such that $|\sin\pi(x+l_0\alpha)|=\inf_{0\leq l\leq q_n-1}|\sin\pi(x+l\alpha)|$, then for some absolute constant $C_1>0$,
$$-C_1\ln q_n\leq \sum_{0\leq l\leq q_n-1, l\neq l_0} \ln|\sin\pi(x+l\alpha)|+(q_n-1)\ln 2\leq C_1\ln q_n$$
\end{lemma}

\begin{lemma}\label{polynomialestimate}\cite{AJ2}
Let $1\leq r\leq [q_{n+1}/q_n]$. If $p(x)$ has essential degree at most $k=rq_n-1$ and $x_0\in\R/{\Z}$, then for some absolute constant $C_2$,
\begin{align*}
\|p(x)\|_0\leq C_2 q_{n+1}^{C_2 r}\sup_{0\leq j\leq k}|p(x_0+j\alpha)|.
\end{align*}
\end{lemma}

\section{Main estimates and proof of Theorem $\ref{dry}$}
\subsection{Almost localization for every $\theta$}
\begin{definition}Let $\alpha\in \R\setminus\Q$, $\theta\in\R$, $\epsilon_0>0$. We say that $k$ is an $\epsilon_0-$resonance of $\theta$ if $\|2\theta-k\alpha\|\leq e^{-\epsilon_0 |k|}$ and $\|2\theta-k\alpha\|=\min_{|l|\leq |k|}\|2\theta-l\alpha\|$.
\end{definition}

\begin{definition}
Let $0=|n_0|<|n_1|<...$ be the $\epsilon_0-$resonances of $\theta$. If this sequence is infinite, we say $\theta$ is $\epsilon_0-$resonant, otherwise we say it is $\epsilon_0-$non-resonant.
\end{definition}

\begin{definition}
We say the extended Harper's model $\{H_{\lambda,\alpha,\theta}\}_{\theta}$ exhibits almost localization if there exists $C_0, C_3, \epsilon_0, \tilde{\epsilon}_0>0$, such that for every solution $\phi$ to $H_{\lambda,\alpha,\theta}\phi=E\phi$ satisfying $\phi(0)=1$ and $|\phi(m)|\leq 1+|m|$, and for every $C_0 (1+|n_j|)<|k|<C_0^{-1}|n_{j+1}|$, we have $|\phi(k)|\leq C_3 e^{-\tilde{\epsilon}_0 |k|}$ (where $n_j$ are the $\epsilon_0-$resonances of $\theta$). 
\end{definition}

\begin{thm}\label{al}
If $\lambda$ belongs to region II,  $\{H_{\hat{\lambda},\alpha,\theta}\}_{\theta}$ is almost localized for every $\alpha\in \mathrm{DC}$.
\end{thm}
\begin{rem}
It is clear from Theorem $\ref{al}$ that almost localization implies localization for non-resonant $\theta$.
\end{rem}

We will actually prove the following explicit lemma:
\begin{lemma}\label{alexplicit}
Let $\lambda$ be in region II. Let $C_4$ be the absolute constant in Lemma $\ref{uniform}$, $\epsilon_1=\epsilon_1(\lambda)$ be as in ($\ref{epsilon1}$), then for any $0<\epsilon_0<\frac{\epsilon_1}{100C_4}$, there exists constant $C_3>0$, which depends on $\lambda, \alpha$ and $\epsilon_0$, so that for every solution $u$ of $H_{\hat{\lambda}, \alpha, \theta}u=Eu$ satisfying $u(0)=1$ and $|u_k|\leq 1+|k|$, if $3(|n_j|+1)<|k|<\frac{1}{3}|n_{j+1}|$, then $|u_k|\leq C_3 e^{-\frac{\epsilon_1}{5} |k|}$, where $\{n_j\}$ are the $\epsilon_0$-resonances of $\theta$.
\end{lemma}

The proof of Lemma \ref{alexplicit} (and thus of Theorem $\ref{al})$ is given in Section 4.

\subsection{Almost reducibility}
\

Let $\lambda$ be in region II. For every $E\in \Sigma_{\lambda}$, let $\theta(E)\in \T$ be given in Theorem $\ref{Etheta}$. Let $0<\epsilon_0<\frac{\epsilon_1}{100C_4}$ and  $\{n_j\}$ be the set of $\epsilon_0-$ resonances of $\theta(E)$.
Then for some positive constants $N_0$, $C$ and $c$, independent of $E$ and $\theta$, we have the following theorem:
\begin{thm}\label{ar}
For any fixed $j$, with $N_0<n=|n_j|+1<\infty$, let $N=|n_{j+1}|$, $L^{-1}=\|2\theta-n_j\alpha\|$. Then there exists $W:\T\rightarrow SL(2,\R)$ analytic such that $|\deg{W}|\leq Cn$, $\|W\|_0\leq CL^C$ and 
$\|W^{-1}(x+\alpha)\tilde{A}_{\lambda, E}(x)W(x)-R_{\mp \theta}\|\leq Ce^{-cN}.$
\end{thm}
\begin{rem}
Notice that this theorem requires $n>N_0$, which is not always ensured when $\theta(E)$ is non-resonant, however in that case we have localization for $H_{\hat{\lambda},\alpha, \theta}$ instead of almost localization.  We will prove Theorem $\ref{ar}$ in Section 5.
\end{rem}

\subsection{Spectral consequences of Almost reducibility}
\

Let $\epsilon_1=\epsilon_1(\lambda)$ and $C_4$ be as in Lemma \ref{alexplicit}.
\begin{thm}\label{rhotheta}
Assume $\alpha\in \mathrm{DC}$. For $\lambda$ in region II, fix $E\in\Sigma_{\lambda}$. Assume $\theta(E) \in \T$ is such that $H_{\hat{\lambda},\alpha,\theta} u=\frac{E}{\lambda_2} u$ has solution satisfying $u_0=1$ and $|u_k|\leq 1$. Let $C$ be the constant in Theorem $\ref{ar}$. Then $\theta(E)$ and $\rho(\alpha,\tilde{A}_{\lambda,E})$ have the following relation:
\begin{itemize}
\item (A)  If $\theta$ is $\epsilon_0$-non-resonant for some $\frac{\epsilon_1}{100C_4}>\epsilon_0>0$, then $2\theta\in \Z\alpha+\Z$ if and only if $2\rho(\alpha, \tilde{A}_{\lambda, E})\in \Z\alpha+\Z$.
\item (B)  If $\theta$ is $\epsilon_0$-resonant for some $\frac{\epsilon_1}{100C_4}>\epsilon_0>0$, then $\rho(\alpha,\tilde{A}_{\lambda, E})$ is $\frac{\epsilon_0}{C+2}$-resonant.
\end{itemize}
\end{thm}

\begin{proof}

(A): When $\theta$ is $\epsilon_0$-non-resonant for some $\frac{\epsilon_1}{100 C_4}>\epsilon_0>0$, Theorem $\ref{al}$ implies $H_{\hat{\lambda},\alpha, \theta}$ has exponentially decaying eigenfunction. Then applying Theorem $\ref{reducible}$ we get $2\theta\in\Z\alpha+\Z$ if and only if $2\rho(\alpha, \tilde{A}_{\lambda, E})\in \Z\alpha+\Z$.

(B): Assume $\theta$ is $\epsilon_0$-resonant for some $\frac{\epsilon_1}{100 C_4}>\epsilon_0>0$.
Fix any $\xi<\frac{\epsilon_0}{2C+2}$, then there exists $C_{\xi}>0$ such that for any $k\neq 0$ we have $\|k\alpha\|\geq C_{\xi}e^{-\xi |k|}$. Now take an $\epsilon_0$-resonance $n_j$ of $\theta$ such that $n=|n_j|>\max{(\frac{-\ln{C_{\xi}/2}}{\epsilon_0-(2C+2)\xi}, N_0)}$. 
Then there exists $|m|\leq Cn$ such that $2\rho(\alpha, \tilde{A}_{\lambda, E})-m\alpha=-2\theta$. Then
$$\|2\rho(\alpha,\tilde{A}_{\lambda, E})-(m-n_j)\alpha\|=\|2\theta-n_j\alpha\|<e^{-\epsilon_0 n}\leq e^{-\frac{\epsilon_0}{C+2}|m-n_j|}.$$
Take any $|l|\leq |m-n_j|$, $l\neq m-n_j$. 
Then
$$\|(l-(m-n_j))\alpha\|\geq C_{\xi} e^{-2\xi |m-n_j|}> 2e^{-\epsilon_0 n}>2\|2\rho(\alpha,\tilde{A}_E)-(m-l_0)\alpha\|.$$
Thus $\|2\rho(\alpha,\tilde{A}_E)-l\alpha\|>\|2\rho(\alpha,\tilde{A}_E)-(m-n_j)\alpha\|$ for any $|l|\leq |m-n_j|$,  $l\neq m-n_j$.
This by definition means $\rho(\alpha,\tilde{A}_{\lambda, E})$ is $\frac{\epsilon_0}{C+2}$-resonant.
\end{proof} $\hfill{} \Box$

Now based on Theorem $\ref{rhotheta}$, we can complete the proof of the dry version of Ten Martini Problem for extended Harper's model in regions I and II.

{\bf Proof of Theorem \ref{dry}}

It is enough to consider $\lambda$ in region II. Let $E\in \Sigma_{\lambda}$ be such that $N_{\lambda}(E)\in \Z\alpha+\Z$. 
We are going to show $E$ belongs to the boundary of a component of $\R\setminus \Sigma_{\lambda}$.
Now by $(\ref{IDSROT})$ we have $2\rho(\alpha,\tilde{A}_{\lambda, E})\in \alpha\Z+\Z$, thus by Theorem $\ref{rhotheta}$, $2\theta(E) \in \alpha\Z+\Z$. By Theorem $\ref{reducible}$, this means there exist $M(x)\in C^{\omega}_h (\T,PSL(2,\R))$ such that 
$M^{-1}(x+\alpha)\tilde{A}_{\lambda, E}(x)M(x)=
\left(\begin{matrix}
\pm 1\ \ &a\\
0\ \ &\pm 1
\end{matrix}
\right).$
Without loss of generality, we assume $M^{-1}(x+\alpha)\tilde{A}_{\lambda, E}(x)M(x)=
\left(\begin{matrix}
1\ \ &a\\
0\ \ & 1
\end{matrix}
\right).$
Let $\tilde{M}(x)=\frac{M(x)}{\sqrt{|c|(x-\alpha)}}$, then
\begin{align*}
\tilde{M}^{-1}(x+\alpha)
\left(
\begin{matrix}
\frac{E-v(x)}{|c|(x)}\ \ &-\frac{|c|(x-\alpha)}{|c|(x)}\\
1\ \ &0
\end{matrix}
\right)
\tilde{M}(x)=
\left(\begin{matrix}
1\ \ &a\\
0\ \ & 1
\end{matrix}
\right).
\end{align*}
Now let $\tilde{M}(x)=
\left(
\begin{matrix}
M_{11}(x)\ \ &M_{12}(x)\\
M_{21}(x)\ \ &M_{22}(x)
\end{matrix}
\right).$
Then $M_{21}(x)=M_{11}(x-\alpha)$ and $M_{22}(x)=M_{12}(x-\alpha)-aM_{11}(x-\alpha)$ amd
\begin{align*}
&\tilde{M}^{-1}(x+\alpha)
\left(
\begin{matrix}
\frac{E+\epsilon-v(x)}{|c|(x)}\ \ &-\frac{|c|(x-\alpha)}{|c|(x)}\\
1\ \ &0
\end{matrix}
\right)
\tilde{M}(x)\\
=&
\left(
\begin{matrix}
1\ \ &a\\
0\ \ &1
\end{matrix}
\right)
+
\epsilon
\left(
\begin{matrix}
M_{11}(x)M_{12}(x)-aM_{11}^2(x)\ \ &M_{12}^2(x)-aM_{11}(x)M_{12}(x)\\
-M^2_{11}(x)\ \ &-M_{11}(x)M_{12}(x)
\end{matrix}
\right).\\
\triangleq &M_0+ \epsilon M_1(x).
\end{align*}
Now we look for $Z_{\epsilon}(x)$ of the form $e^{\epsilon Y(x)}$ such that
\begin{align*}
Z^{-1}_{\epsilon}(x+\alpha) (M_0+\epsilon M_1(x)) Z_{\epsilon}(x)=M_0+\epsilon [M_1]+O(\epsilon^2).
\end{align*}
We then just need to solve the equation:
\begin{align*}
(I-\epsilon Y(x+\alpha)+O(\epsilon^2))(M_0+\epsilon M_1(x)) (I+\epsilon Y(x)+O(\epsilon^2))=M_0+\epsilon [M_1]+O(\epsilon^2).
\end{align*}
It is sufficient to solve the coholomogical equation:
\begin{align*}
Y(x+\alpha)M_0-M_0 Y(x)=M_1(x)-[M_1],
\end{align*}
which is guaranteed by the Diophantine condition on $\alpha$.
Thus 
\begin{align*}
&(M(x+\alpha)Z_{\epsilon}(x+\alpha))^{-1}
\tilde{A}_{\lambda, E}(x)
(M(x)Z_{\epsilon}(x))
\\
=&
\left(
\begin{matrix}
1+\epsilon [M_{11}M_{12}]-a\epsilon [M_{11}^2]\ \ &a+\epsilon [M_{12}^2]-a\epsilon [M_{11}M_{12}]\\
-\epsilon [M_{11}^2] \ \ &1-\epsilon [M_{11}M_{12}]
\end{matrix}
\right)+O(\epsilon^2)\\
\triangleq &M_{\epsilon}+O(\epsilon^2).
\end{align*}
Notice that $\tilde{A}_{\lambda,E}$ is uniformly hyperbolic iff $\mathrm{Trace}(M_{\epsilon})>2$ which is fulfilled when $-a\epsilon [M_{11}^2]>0$. Thus for $\epsilon$ small, satisfying $-a\epsilon [M_{11}^2]>0$, $E+\epsilon\notin \Sigma_{\lambda}$, which means this spectral gap is open.
$\hfill{} \Box$

\section{Almost localization in region I}
In this section we will prove Lemma $\ref{alexplicit}$. 
For fixed $\lambda$ in region II and $E$, let $D_{\hat{\lambda}, E} (\theta)=c_{\hat{\lambda}}(\theta)A_{\hat{\lambda},E}(\theta)$, where
$c_{\hat{\lambda}}(\theta)=\frac{\lambda_3}{\lambda_2}e^{-2\pi i (\theta+\frac{\alpha}{2})}+\frac{1}{\lambda_2}+\frac{\lambda_1}{\lambda_2}e^{2\pi i (\theta+\frac{\alpha}{2})}$. 
Regarding the Lyapunov exponent, we recall the following result in \cite{Jm},
\begin{align*}
L(\alpha, A_{\hat{\lambda}, E})=L(\alpha, D_{\hat{\lambda},E})-\int_{\T}\ln{|c_{\hat{\lambda}}(\theta)|}\mathrm{d}\theta \triangleq \tilde{L}-\int \ln{|c_{\hat{\lambda}}|}>0,
\end{align*}
where $\tilde{L}=\ln{\frac{\lambda_2+\sqrt{\lambda_2^2-4\lambda_1\lambda_3}}{2\lambda_2}}$ and $\int \ln{|c_{\hat{\lambda}}|}=\ln{\frac{\max{(\lambda_1+\lambda_3,1)}+\sqrt{\max{(\lambda_1+\lambda_3,1)}^2-4\lambda_1\lambda_3}}{2\lambda_2}}$.

{\bf Proof of of Lemma $\ref{alexplicit}$}
\

Suppose $u$ is a solution satisfying the condition of Lemma $\ref{alexplicit}$. 
For an interval $I=[x_1, x_2]$, let $\Gamma_I$ be the coupling operator between $I$ and $\Z\setminus I$:
\begin{align*}
\Gamma_I(i,j)=
\left\lbrace
\begin{matrix}
&\tilde{c}(\theta+(x_1-1)\alpha),\ \ (i,j)=(x_1, x_1-1)\\
&c(\theta+(x_1-1)\alpha),\ \ (i,j)=(x_1-1, x_1)\\
&\tilde{c}(\theta+x_2\alpha),\ \ \ \ \ \ \ \ \ (i,j)=(x_2+1, x_2)\\
&c(\theta+x_2\alpha),\ \ \ \ \ \ \ \ \ (i,j)=(x_2, x_2+1)\\
&0\ \ \ \ \ \ \ \ \ \ \ \ \ \ \ \ \ \ \ \ \ \ \ \ \ \ \ \ \ \ \ \ \ \ \mathrm{otherwise}.
\end{matrix}
\right.
\end{align*} 
Let $H_I=R_IH_{\hat{\lambda}, \theta} R_I^*$ be the restricted operator of $H_{\hat{\lambda} ,\theta}$ to $I$. Then for $x\in I$, we have $(H_I+\Gamma_I-E)u(x)=0$. Thus $u(x)=G_I\Gamma_Iu(x)$, where $G_I=(E-H_I)^{-1}$. By matrix multiplication:

\begin{align*}
u(x)&=\sum_{y\in I,(y,z)\in \Gamma_I}G_I(x,y)\Gamma_I (y,z)u(z)\\
&=\tilde{c}(\theta+(x_1-1)\alpha) G_I(x,x_1)u (x_1-1)+ c(\theta+x_2\alpha) G_I(x,x_2)u(x_2+1).
\end{align*}

Let us denote $P_k(\theta)=\det{(E-H_{[0,k-1]}(\theta))}$. Then the $k-$step matrix $D_{\hat{\lambda}, E, k}(\theta)$ satisfies:
\begin{align*}
D_{\hat{\lambda}, E, k}(\theta)=
\left(
\begin{matrix}
P_k(\theta)\ \ &-\tilde{c}(\theta-\alpha)P_{k-1}(\theta+\alpha)\\
c(\theta+(k-1)\alpha)P_{k-1}(\theta)\ \ &-\tilde{c}(\theta-\alpha)c(\theta+(k-1)\alpha)P_{k-2}(\theta+\alpha)
\end{matrix}
\right).
\end{align*}
This relation between $P_k(\theta)$ and $D_{\hat{\lambda}, E, k}(\theta)$ gives a general upper bound of $P_k(\theta)$ in terms of $\tilde{L}$. Indeed by Lemma $\ref{uniformupp}$, 
for any $\epsilon>0$ there exists $C(\epsilon)>0$ so that 
\begin{align*}
|P_n(\theta)|\leq C(\epsilon) e^{(\tilde{L}+\epsilon)n}\ \ \mathrm{for}\ \mathrm{any}\ n\in \N.
\end{align*}

By Cramer's rule:
\begin{align*}
&|G_I(x_1,y)|=\prod_{j=x_1}^{y-1}|c(\theta+j\alpha)||\frac{\det{(E-H_{[y+1,x_2]}(\theta))}}{\det{(E-H_I(\theta))}}|=\prod_{j=x_1}^{y-1}|c(\theta+j\alpha)||\frac{P_{x_2-y}(\theta+(y+1)\alpha)}{P_k(\theta+x_1\alpha)}|,\\
&|G_I(y,x_2)|=\prod_{j=y+1}^{x_2}|c(\theta+j\alpha)||\frac{\det{(E-H_{[x_1, y-1]}(\theta))}}{\det{(E-H_I(\theta))}}|=\prod_{j=y+1}^{x_2}|c(\theta+j\alpha)||\frac{P_{y-x_1}(\theta+x_1\alpha)}{P_k(\theta+x_1\alpha)}|.
\end{align*}

Notice that $P_k(\theta)$ is an even function about $\theta+\frac{k-1}{2}\alpha$, it can be written as a polynomial of degree $k$ in $\cos{2\pi(\theta+\frac{k-1}{2}\alpha)}$. Let $P_k(\theta)=Q_k(\cos{2\pi (\theta+\frac{k-1}{2}\alpha)})$. Let $M_{k,r}=\{\theta\in \T,\ |Q_k(\cos{2\pi \theta})|\leq e^{(k+1)r}\}$.

\begin{definition}
Fix $m>0$. A point $y\in\Z$ is called $(k,m)-$regular if there exists an interval $[x_1, x_2]$ containing $y$, where $x_2=x_1+k-1$ such that
$$|G_I(y, x_i)|\leq e^{-m|y-x_i|} \ \mathrm{and}\ \mathrm{dist}(y, x_i)\geq  \frac{1}{3} k\ \mathrm{for}\ i=1,2,$$
otherwise $y$ is called $(k,m)-$singular.
\end{definition}

\begin{lemma}
Suppose $y\in\Z$ is $(k,\tilde{L}-\int \ln{|c_{\hat{\lambda}}|}-\rho)-$singular. Then for any $\epsilon>0$ and any $x\in\Z$ satisfying $y-\frac{2}{3}k\leq x\leq y-\frac{1}{3} k$, we have $\theta+(x+\frac{1}{2}(k-1))\alpha$ belongs to $M_{k, \tilde{L}-\frac{1}{3} \rho+\epsilon}$ for $k>k(\lambda,\epsilon,\rho)$.
\end{lemma}

\begin{proof}
Suppose there exists $\epsilon>0$ and $x_1$: $y-(1-\delta)k\leq x_1\leq y-\delta k$, 
such that $\theta+(x_1+\frac{1}{2}(k-1))\alpha$ does not belong to $M_{k,\tilde{L}-\frac{1}{3} \rho+\epsilon}$, that is  $|P_k(\theta+x_1\alpha)|>e^{(k+1)(\tilde{L}-\rho\delta+\epsilon)}$,
\begin{align*}
|G_I(x_1, y)| & \leq \prod_{j=x_1}^{y-1}|c_{\hat{\lambda}}(\theta+j\alpha)|e^{(k-|x_1-y|)(\tilde{L}+\epsilon)} e^{-(k+1)(\tilde{L}-\frac{1}{3}\rho+\epsilon)}\\
                        &<e^{-(\tilde{L}-\int \ln{|c_{\hat{\lambda}}|}-\rho)|y-x_1|}\ \ \mathrm{for}\ k>k(\lambda,\epsilon,\rho).
\end{align*}
Similarly
$$|G_I(x_2,y)|\leq e^{-(\tilde{L}-\int \ln{|c_{\hat{\lambda}}|}-\rho)|y-x_2|}.$$

\begin{center}
\begin{tikzpicture}
    \draw[->] (-7,0) -- (7,0) node [right] {$x$} ;

   \draw [-](4.2,0) node [below] {$y$};
   \draw [-](2.0,0) node [below] {$y-(\frac{1}{2}-\delta)k$};
   \draw [-](6.4,0) node [below] {$y+(\frac{1}{2}-\delta)k$};

   \draw [-](-6.0,0) node [below] {$y-(1-\delta)k$};
   \draw [-](-1.6,0) node [below] {$y-\delta k$};

   \draw[loosely dashed,|<->|] (2.0,0.2) -- (6.4,0.2)node[pos=0.5,above]{$x+\frac{1}{2}(k-1)\alpha$};
   \draw[loosely dashed,|<->|] (-6,0.2) -- (-1.6,0.2)node[pos=0.5,above]{$x$};

\end{tikzpicture}
\end{center}
\end{proof} $\hfill{} \Box$

\begin{definition}
We say that the set $\{\theta_1, ..., \theta_{k+1}\}$ is $\gamma-$uniform if
$$\max_{x\in [-1, 1]}\max_{i=1,...,k+1}\prod_{j=1, j\neq i}^{k+1}\frac{|x-\cos{2\pi\theta_j}|}{|\cos{2\pi\theta_i}-\cos{2\pi\theta_j}|}<e^{k\gamma}$$
\end{definition}

\begin{lemma}\label{gamma1}
Let $\gamma_1<\gamma$. If $\theta_1, ..., \theta_{k+1}\in M_{k, \tilde{L}-\gamma}$, then $\{\theta_1, ..., \theta_{k+1}\}$ is not $\gamma_1-$uniform for $k>k(\gamma, \gamma_1)$.
\end{lemma}

\begin{proof}
Otherwise, using Lagrange interpolation form we can get
$|Q_k(x)|<e^{k\tilde{L}}$ for all $x\in [-1,1]$. 
This implies $|P_k(x)|<e^{k\tilde{L}}$ for all $x$. 
But by Herman's subharmonic function argument, $\int_{\R/\Z}\ln|P_k(x)|\mathrm{d}x \geq k\tilde{L}$. This is impossible.
\end{proof}  $\hfill{} \Box$
\

\

Now take $\xi$ and $\epsilon_0$ such that $0<1000\xi<\epsilon_0$. Then for $|n_{j+1}|>N(\xi)$ we have
\begin{align*}
2e^{-4\xi |n_{j+1}|} \leq  C_{\xi} e^{-2\xi|n_{j+1}|}\leq \|(n_{j+1}-n_j)\alpha\|= \|n_{j+1}\alpha-2\theta+2\theta-n_j\alpha\|\leq 2\|2\theta-n_j\alpha\|\leq 2e^{-\epsilon_0 |n_j|},
\end{align*}
which yields that 
\begin{equation}\label{250}
|n_{j+1}|>\frac{\epsilon_0}{4\xi} |n_j|>250|n_j|.
\end{equation}

Without loss of generality, assume $3(|n_j|+1)<y<\frac{|n_{j+1}|}{3}$ and $y>N(\xi)$. Select $n$ such that $q_n\leq \frac{y}{8}< q_{n+1}$ and let $s$ be the largest positive integer satisfying $sq_n\leq \frac{y}{8}$. Set $I_1, I_2\subset \Z$ as follows
\begin{align*}
I_1=[1-2sq_n, 0]\ and\ I_2=[y-2sq_n+1, y+2sq_n],\ &\mathrm{if}\ n_j<0      \\
I_1=[0, 2sq_n-1]\ and\ I_2=[y-2sq_n+1, y+2sq_n],\  &\mathrm{if}\ n_j\geq 0
\end{align*}

\begin{lemma}\label{uniform}
Let $\theta_j=\theta+j\alpha$, then set $\{\theta_{j}\}_{j\in I_1\cup I_2}$ is $C_4\epsilon_0+C_4\xi-$uniform for some absolute constant $C_4$ and $y>y(\alpha, \epsilon_0, \xi)$.
\end{lemma}

\begin{proof}
Without loss of generality, we assume $n_j>0$.
Take $x=\cos{2\pi a}$.
Now it suffices to estimate
\begin{align*}
\sum_{j\in I_1\cup I_2,\ j\neq i}\left( \ln{|\cos{2\pi a}-\cos{2\pi \theta_j}|}-\ln{|\cos{2\pi \theta_i}-\cos{2\pi \theta_j}|}\right) \triangleq \sum_1-\sum_2.
\end{align*}
Lemma $\ref{smallest}$ reduces this problem to estimating the minimal terms.

First we estimate $\sum_1$:
\begin{align*}
\sum_{1}
&=\sum_{j\in I_1\cup I_2, j\neq i} \ln|\cos{2\pi a}-\cos{2\pi\theta_{j}}|\\
&=\sum_{j\in I_1\cup I_2, j\neq i} \ln|\sin{\pi(a+\theta_{j})}|+\sum_{j\in I_1\cup I_2, j\neq i} \ln|\sin{\pi(a-\theta_{j})}|+(6sq_n-1)\ln 2\\
&\triangleq\sum_{1,+}+\sum_{1,-}+(6sq_n-1)\ln 2.\\
\end{align*}

We cut $\sum_{1,+}$ or $\sum_{1,-}$ into $6s$ sums and then apply Lemma $\ref{smallest}$, we get that for some absolute constant $C_1$:
\begin{align*}
\sum_{1}\leq -6sq_n\ln 2+C_1 s\ln q_n.
\end{align*}

Next, we estimate $\sum_2$.
\begin{align*}
\sum_2
&=\sum_{j\in I_1\cup I_2, j\neq i}\ln |\cos 2\pi\theta_j-\cos 2\pi\theta_i|\\
&=\sum_{j\in I_1\cup I_2, j\neq i} \ln|\sin{\pi(2\theta+(i+j)\alpha)}|+\sum_{j\in I_1\cup I_2, j\neq i} \ln|\sin{\pi(i-j)\alpha}|+(6sq_n-1)\ln 2\\
&\triangleq \sum_{2,+} + \sum_{2,-} +(6sq_n-1)\ln 2.
\end{align*}
We need to carefully estimate the minimal terms. For $\sum_{2,+}$, we use the property of resonant set; and for $\sum_{2,-}$, we use the Diophantine condition on $\alpha$.

For any $0<|j|< q_{n+1}$ , we have $\|j\alpha\|\geq \|q_n\alpha\|\geq  C_{\xi}e^{-\xi q_n}$. Therefore
$$\max(\ln|\sin{x}|, \ln|\sin(x+\pi j\alpha)|)\geq -2\xi q_n\ \ \mathrm{for}\ y>y(\alpha, \xi).$$
This means in any interval of length $sq_n$, there can be at most one term which is less than $-2\xi q_n$. Then there can be at most $6$ such terms in total. 

For the part $\sum_{2,-}$, since $\|(i-j)\alpha\|\geq C_{\xi}e^{-\xi |i-j|}\geq e^{-20\xi sq_n}$, these 6 smallest terms must be bounded by $-20\xi sq_n$ from below. Hence $\sum_{2,-}\geq -6sq_n\ln 2-C \xi sq_n-Cs\ln{q_n}$ for $y>y(\xi)$ and some absolute constant $C$.

For the part $\sum_{2,+}$,  notice $|i+j|\leq 2y+4sq_n<3y<|n_{j+1}|$ and $i+j>0>-n_j$. 
Suppose $\|2\theta+k_0\alpha\|=\min_{j\in I_1\cup I_2}\|2\theta+(i+j)\alpha\|\leq e^{-100\epsilon_0 sq_n}<e^{-\epsilon_0 |k_0|}$. Then for any $|k|\leq |k_0| \leq 40 sq_n$ (including $|n_j|$),
\begin{align*}
\|2\theta-k\alpha\|\geq \|(k+k_0)\alpha\|-\|2\theta+k_0\alpha\| >\|2\theta+k_0\alpha\|\ \ \mathrm{for}\ y>y(\alpha, \epsilon_0,\xi).
\end{align*}
This means $-k_0$ must be a $\epsilon_0-$resonance, therefore $|k_0|\leq |n_{j-1}|$. Then
$$\|2\theta-n_j\alpha\|\geq \|(n_j+k_0)\alpha\|-\|2\theta+k_0\alpha\|\geq C_{\xi}e^{-12\xi sq_n}-e^{-100 \epsilon_0 sq_n}>e^{-100 \epsilon_0 sq_n}\geq \|2\theta+k_0\alpha\|$$
leads to a contradiction.
Thus the smallest terms must be greater than $-100\epsilon_0 sq_n$. We can bound $\sum_{2,+}$ by $-6sq_n\ln 2-600\epsilon_0 sq_n-12\xi sq_n-Cs\ln{q_n}$ from below. Therefore $\sum_{2}\geq -6sq_n\ln 2-C\epsilon_0 sq_n-C \xi sq_n-Cs\ln{q_n}$. Thus the set $\{\theta_j\}_{j\in I_1\cup I_2}$ is $C_4\epsilon_0+C_4\xi-$uniform for  $y>y(\alpha, \epsilon_0, \xi)$ and some absolute constant $C_4$.
\end{proof} $\hfill{} \Box$

Now let $C_4$ be the absolute constant in Lemma $\ref{uniform}$. Choose $0<1000\xi<\epsilon_0<\frac{\epsilon_1}{100C_4}$. 
Combining Lemma $\ref{gamma1}$ and Lemma $\ref{uniform}$, we know that when $y>y(\alpha, \epsilon_0, \xi)$, $\{\theta_j\}_{j\in I_1\cup I_2}$ can not be inside the set $M_{6sq_n-1, \tilde{L}-2C_4 \epsilon_0}$ at the same time. 
Therefore $0$ and $y$ can not be $(6sq_n-1, \tilde{L}-\int \ln{|c_{\hat{\lambda}}|}-9C_4 \epsilon_0)$ at the same time. 
However $0$ is $(6sq_n-1, \tilde{L}-\int \ln{|c_{\hat{\lambda}}|}-9C_4 \epsilon_0)-$singular given $n$ large enough. Therefore 
$$\{\theta_j\}_{j\in I_1}\subset M_{6sq_n-1, \tilde{L}-2C_4 \epsilon_0}.$$
Thus $y$ must be $(6sq_n-1,\tilde{L}-\int \ln{|c_{\hat{\lambda}}|}-9C_4 \epsilon_0)-$regular. 
This implies
\begin{align*}
|u(y)|\leq e^{-(\tilde{L}-\int \ln{|c_{\hat{\lambda}}|}-9C_4 \epsilon_0)\frac{1}{4}|y|}<e^{-\frac{\epsilon_1}{5} |y|}\ \ \mathrm{for}\ |y|\geq y(\lambda, \alpha, \epsilon_0, \xi).
\end{align*}
Thus there exists $C_3=C_{\lambda, \alpha, \epsilon_0, \xi}$ such that 
$|u(y)|\leq C_3 e^{-\frac{\epsilon_1}{5}|y|}$ for any $3|n_j|\leq |y|\leq \frac{1}{3}|n_{j+1}|$ and $j\in \N$.

\

\section{Almost reducibility in region II}

{\bf Proof of Theorem $\ref{ar}$}
\

For any $E\in \Sigma_{\lambda}$, take $\theta(E)$ and $\{u_k\}$ as in Theorem $\ref{Etheta}$. Let $\epsilon_1$ be as in (\ref{epsilon1}),
$C_4$ be the absolute constant from Lemma $\ref{uniform}$, and $C_2$ be the absolute constant from Lemma $\ref{polynomialestimate}$. Fix $\max{(32C_2 \xi, 1000\xi)}<\epsilon_0<\min{(\frac{\epsilon_1}{200}, \frac{\epsilon_1}{100C_4})}$. By Lemma $\ref{alexplicit}$, there exists $C$ depending on $\lambda$ and $\alpha$ such that for any $3|n_j|<|k|<\frac{1}{3}|n_{j+1}|$, we have $|u_k|\leq C e^{-\frac{\epsilon_1}{5} |k|}$.

For any $n$, $9|n_j|<n<\frac{1}{9}|n_{j+1}|$, of the form
\begin{equation}\label{nform}
n=rq_m-1<q_{m+1}.\ \footnote{The existence of such $n$ comes from ($\ref{250}$).}
\end{equation}
Let $u(x)=u^{I}(x)=\sum_{k\in I} u_k e^{2\pi i kx}$ with $I=[-[\frac{n}{2}], [\frac{n}{2}]]=[x_1,x_2]$. Define
\begin{align*}
U(x)=
\left(
\begin{matrix}
e^{2\pi i\theta}u(x)\\
u(x-\alpha)
\end{matrix}
\right).
\end{align*}
Let $A(\theta)=A_{\lambda, E}(\theta)$.
By direct computation:
\begin{align*}
A(x)U(x)=
e^{2\pi i\theta}U(x+\alpha)+
\left(
\begin{matrix}
g(x)\\
0
\end{matrix}
\right)
\triangleq
e^{2\pi i \theta} U(x+\alpha)+G(x).
\end{align*}
The Fourier coefficients of $g(x)$ are possibly nonzero only at four points $x_1$, $x_2$, $x_1-1$ and $x_2+1$. Since $|u_k|\leq C_1 e^{-\frac{\epsilon_1}{5} |k|}$ when $3|n_j|<|k|<\frac{1}{3}|n_{j+1}|$, we know that ${\|G(x)\|}_{\frac{\epsilon_1}{20\pi}}\leq C_1 e^{- \frac{\epsilon_1}{20} n}$.

Combining Lemma $\ref{LEregion2}$ and $\ref{uniformupp}$, we have exponential control of the growth of the transfer matrix, for any $\delta>0$ there exists $C_{\delta}>0$ such that
\begin{align*}
\|\tilde{A}_k(x)\|_{\frac{\epsilon_1}{2\pi}}\leq C_{\delta}e^{\delta |k|},\ \ \mathrm{for}\ \mathrm{any}\ k.
\end{align*}
With some effort we are able to get the following significantly improved upper bound:

\begin{thm}\label{polycontrol}
For some $C>0$ depending on $\lambda$ and $\alpha$,
\begin{align*}
\|\tilde{A}_k(x)\|_{\T}\leq C{(1+|k|)}^{C}.
\end{align*}
\end{thm}

\begin{proof}

Let $\tilde{U}(x)=Q(x)U(x)$, $\tilde{G}(x)=Q(x+\alpha)G(x)$, where $Q=Q_{\lambda}$ is given in $(\ref{conjugate})$. 
Since $$\max{(\|Q(x)\|_{\frac{\epsilon_1}{20\pi}}, \|Q^{-1}(x)\|_{\frac{\epsilon_1}{20\pi}})}  \leq C,$$ 
we have
\begin{align*}
\tilde{A}(x)\tilde{U}(x)=e^{2\pi i \theta} \tilde{U}(x+\alpha)+\tilde{G}(x),
\end{align*}
where $\|\tilde{G}(x)\|_{\frac{\epsilon_1}{20\pi}}\leq C e^{-\frac{\epsilon_1}{20}n}$.

\begin{lemma}\label{Ulower}
Let $C_2$ be the constant from Lemma $\ref{polynomialestimate}$, then for any $\delta$, $2C_2\xi<\delta<\frac{\epsilon_0}{16}$, we have
\begin{align*}
\inf_{|\mathrm{Im}{(x)}|\leq\frac{\epsilon_1}{20\pi}}\|\tilde{U}(x)\|\geq e^{-2\delta n},
\end{align*}
for $n>n(\alpha, \delta)$.
\end{lemma}

\begin{proof}
We will prove the statement by contradiction. Suppose for some $x_0\in \{|\mathrm{Im}{(x)}|\leq \frac{\epsilon_1}{20\pi}\}$ we have $\|\tilde{U}(x_0)\|< e^{-2\delta n}$.
Notice that for any $l\in\N$,
\begin{align*}
e^{2\pi i l\theta}\tilde{U}(x_0+l\alpha)=\tilde{A}_l(x_0) \tilde{U}(x_0)-\sum_{m=1}^l e^{2\pi i (m-1)\theta}\tilde{A}_{l-m}(x_0+m\alpha) \tilde{G}(x_0+(m-1)\alpha).
\end{align*}
This implies for $n>n(\delta)$ large enough and for any $0\leq l\leq n$, $\|\tilde{U}(x_0+l\alpha)\| \leq e^{-\delta n}$, thus $\|u(x_0+l\alpha)\|\leq C_{\delta} e^{-\delta n}$.
By Lemma $\ref{polynomialestimate}$, $\|u(x+i\mathrm{Im}(x_0))\|_{\T}\leq C_2 C_{\delta}  e^{C_2\xi n}e^{-\delta n}\leq e^{-\frac{\delta}{2} n}$. 
This contradicts with $\int_{\T} u(x+i\mathrm{Im}(x_0)) \mathrm{d}x=u_0=1$.
\end{proof} $\hfill{} \Box$

\begin{lemma}\label{columnmatrix}\cite{A2}
Let $V$ : $\T\rightarrow \C^2$ be analytic in $|\mathrm{Im}(x)|<\eta$. Assume that $\delta_1<\|V(x)\|<\delta_2^{-1}$ holds on $|\mathrm{Im}(x)|<\eta$. 
Then there exists $M$ : $\T\rightarrow SL(2, \C)$  analytic on $|\mathrm{Im}(x)|<\eta$ with first column $V$ and $\|M\|_\eta\leq C\delta_1^{-2}\delta_2^{-1}(1-\ln(\delta_1\delta_2))$.
\end{lemma}
Applying Lemma $\ref{columnmatrix}$, let $M(x)$ be the matrix with first column $\tilde{U}(x)$. Then $e^{-2\delta n}\leq \|\tilde{U}(x)\|_{\frac{\delta}{\pi}}\leq e^{\delta n}$ and hence $\|M(x)\|_{\frac{\delta}{\pi}}\leq Ce^{6\delta n}$. 
Therefore
\begin{align*}
M^{-1}(x+\alpha)\tilde{A}(x)M(x)=
\left(
\begin{matrix}
e^{2\pi i \theta}   &0\\
0                   &e^{-2\pi i\theta}
\end{matrix}
\right)
+
\left(
\begin{matrix}
\beta_1(x)          &b(x)\\
\beta_3(x)          &\beta_4(x)
\end{matrix}
\right)
\end{align*}
where $\|\beta_1(x)\|_{\frac{\delta}{\pi}},\ \|\beta_3(x)\|_{\frac{\delta}{\pi}},\ \|\beta_4(x)\|_{\frac{\delta}{\pi}}\leq C e^{-\frac{\epsilon_1}{40}n}$, and  $\|b(x)\|_{\frac{\delta}{\pi}}\leq C e^{13\delta n}$. Let
\begin{align*}
\Phi(x)=M(x)
\left(
\begin{matrix}
e^{\frac{\epsilon_1}{160}n}     &0\\
0                                               &e^{-\frac{\epsilon_1}{160}n}
\end{matrix}
\right).
\end{align*} 
Then we would have:

\begin{align*}
{\Phi(x+\alpha)}^{-1}\tilde{A}(x)\Phi(x)=
\left(
\begin{matrix}
e^{2\pi i\theta}    &0\\
0                   &e^{-2\pi i\theta}
\end{matrix}
\right)+
H(x),
\end{align*}
where $\|H(x)\|_{\frac{\delta}{\pi}}\leq Ce^{-\frac{\epsilon_1}{160}n}$, and $\|\Phi(x)\|_{\frac{\delta}{\pi}}\leq Ce^{\frac{\epsilon_1}{80}n}$. Thus
\begin{align*}
\sup_{0\leq s\leq e^{\frac{\epsilon_1}{320}n}}\|\tilde{A}_s(x)\|_{\T}\leq e^{\frac{\epsilon_1}{20}n}
\end{align*}
for $n\geq n(\lambda, \alpha)$ satisfying ($\ref{nform}$).
For $s$ large, there always exists $9|n_j|<n<\frac{1}{9}|n_{j+1}|$ satisfying $(\ref{nform})$ such that $cn \leq\frac{320}{\epsilon_1}\ln{s}\leq n$ with some absolute constant $c$. Thus there exists $C$ depending on $\lambda$ and $\alpha$ such that $\|\tilde{A}_k(x)\|_{\T}\leq C(1+|k|)^{C}$.
$\hfill{} \Box$
\end{proof}

Now we come back to the proof of Theorem $\ref{ar}$. Fix some $n=|n_j|$, and $N=|n_{j+1}|$. Let $u(x)=u^{I_2}(x)$ with $I_2=[-[\frac{N}{9}], [\frac{N}{9}]]$ and $U(x)=\left(\begin{matrix}e^{2\pi i\theta}u(x)\\ u(x-\alpha)\end{matrix}\right)$. Then
\begin{align*}
A(x)U(x)=e^{2\pi i \theta}U(x+\alpha)+G(x) \ \ with\ \ {\|G(x)\|}_{\frac{\epsilon_1}{20\pi}}\leq C e^{-\frac{\epsilon_1}{90}N}.
\end{align*}
Define $U_0(x)=e^{\pi i n_j x}U(x)$. Notice that if $n_j$ is even, then $U_0(x)$ is well-defined on $\T$, otherwise $U_0(x+1)=-U_0(x)$.
\begin{align*}
\tilde{A}(x)\tilde{U}_0(x)&=e^{2\pi i \tilde{\theta}} \tilde{U}_0(x+\alpha)+H(x),
\end{align*}
where $\tilde{\theta}=\theta-\frac{n_j}{2}\alpha$, $\tilde{U}_0(x)=Q(x)U_0(x)$ and $\|H(x)\|_{\frac{\epsilon_1}{20\pi}}\leq Ce^{-\frac{\epsilon_1}{100}N}$.
Consider the matrix $W(x)$ with $\tilde{U}_0(x)$ and $\overline{\tilde{U}_0(x)}$ being its two columns. Then
\begin{align*}
\tilde{A}(x)W(x)=W(x+\alpha)
\left(
\begin{matrix}
e^{2\pi i\tilde{\theta}}   &0\\
0                                  &e^{-2\pi i \tilde{\theta}}
\end{matrix}
\right)
+\tilde{H}(x).
\end{align*}

\begin{thm}\label{lowerdet}Let $L^{-1}=\|2\theta-n_j\alpha\|$. Then for $n>N_0(\lambda, \alpha)$ we have
$$
|\det{W(x)}|
\geq L^{-4C}\ \ \mathrm{for}\ \mathrm{any}\ x\in \T,
$$ where $C$ is the constant appeared in Theorem $\ref{polycontrol}$.
\end{thm}

\begin{proof} First, we fix $\xi_1<\frac{\epsilon_0}{1600}$ so that $\|k\alpha\|\geq C_{\xi_1}e^{-\xi_1 |k|}$ for any $k\neq 0$. We have the following estimate about $L$:
\begin{lemma} $e^{\epsilon_0 n}\leq L\leq e^{4\xi_1 N}$.
\begin{align*}
e^{-2\xi_1 N}\leq \|(n_{j+1}-n_j)\alpha\|\leq 2\|n_j\alpha-2\theta\|=2L^{-1}\leq 2 e^{-\epsilon_0 n}\ \ \mathrm{for}\ n\geq N(\xi_1).
\end{align*}
\end{lemma}
Now we prove by contradiction. Suppose there exists $\kappa$ and $x_0\in \T$ such that $\|\tilde{U}_0(x_0)-\kappa \overline{\tilde{U}_0(x_0)}\|<L^{-4C}$. 
Then 
\begin{align*}
  &\|\tilde{U}_0(x_0+l\alpha)e^{2\pi i l\tilde{\theta}}-\kappa \overline{\tilde{U}_0(x_0+l\alpha)}e^{-2\pi il\tilde{\theta}}\|\\
\leq &\|\sum_{m=0}^{l-1}\tilde{A}_{l-m}(x_0+m\alpha)H(x_0+m\alpha)-\kappa \sum_{m=0}^{l-1}\tilde{A}_{l-m}(x_0+m\alpha)\overline{H(x_0+m\alpha)}\|+\|A_l(x_0)\| L^{-4C} \\
\leq &C L^{2C} e^{-\frac{\epsilon_1}{100}N}+CL^{-2C}<L^{-C}.
\end{align*}
for $0\leq |l|\leq L^2$. 
If we take $j=\frac{L}{4}$, then 
\begin{align}
\|\tilde{U}_0(x_0+\frac{L}{4}\alpha)+\kappa \overline{\tilde{U}_0(x_0+\frac{L}{4}\alpha)}\|<L^{-1}.
\end{align} 
Next since $\|U_0(x)\|_{\T}\leq n$, we have $\|\tilde{U}_0(x)\|_{\T}\leq C n$. 
Thus 
\begin{align*}
\|\tilde{U}_0(x_0+l\alpha)-\kappa \overline{\tilde{U}_0(x_0+l\alpha)}\|<L^{-\frac{1}{3}}\ \ \mathrm{for}\ 0\leq |l|\leq L^{\frac{1}{2}}.
\end{align*} 
For any analytic function $f(x)=\sum_{k\in \Z} \hat{f}_k e^{2\pi i kx}$, define $f_{[-m,m]}(x)=\sum_{|k|\leq m} \hat{f}_k e^{2\pi i k x}$. 
For any column vector $V(x)=\left(\begin{matrix} v^{(1)}(x)\\ v^{(2)}(x)\end{matrix}\right)$, let $V_{[-m,m]}(x)=\left(\begin{matrix} v^{(1)}_{[-m,m]}(x)\\ v^{(2)}_{[-m,m]}(x)\end{matrix}\right)$.
Now let us define $\tilde{U}_0^{[9n]}(x)=Q(x)e^{\pi i n_j x}U_{[-9n, 9n]}(x)$. Then 
\begin{align*}
\|\tilde{U}_0^{[9n]}(x)-\tilde{U}_0(x)\|_{\T}\leq C e^{-\frac{9}{5}\epsilon_1 n}.
\end{align*} 
Consider $[e^{-\pi i n_jx}\tilde{U}_0^{[9n]}(x)]_{[-18n, 18n]}(x)e^{\pi i n_jx}$. This function differs from a polynomial with essential degree $36n$ only by a multiple of $e^{\pi i n_jx}$. 
Notice that $Q(x)$ is analytic in $\{x:|\mathrm{Im}(x)|\leq \frac{\epsilon_1}{4\pi}\}$, thus $|\hat{Q}(k)|\leq C e^{-\frac{\epsilon_1}{2} |k|}$. 
Then 
\begin{align*}
|\widehat{e^{-\pi in_jx}\tilde{U}_0^{[9n]}}(k)|\leq \sum_{|m|\leq 9n}|\hat{Q}(k-m)\hat{U}(m)|\leq C n e^{-\frac{\epsilon_1}{2} (|k|-9n)}\ \ \mathrm{for}\ |k|\geq 18n.
\end{align*}
Thus 
\begin{align*}
\|e^{-\pi i n_jx}\tilde{U}_0^{[9n]}(x)-[e^{-\pi i n_jx}\tilde{U}_0^{[9n]}]_{[-18n, 18n]}(x)\|_{\T}\leq e^{-4\epsilon_1 n},
\end{align*}
\begin{align*}
\|\tilde{U}_0(x)-[e^{-\pi i n_jx}\tilde{U}_0^{[9n]}]_{[-18n, 18n]}(x)e^{\pi i n_jx}\|_{\T} \leq e^{-4\epsilon_1 n}.
\end{align*}
Hence 
\begin{align*}
   &\|[e^{-\pi i n_jx}\tilde{U}_0^{[9n]}]_{[-18n, 18n]}(x_0+l\alpha)e^{2\pi i n_j(x_0+l\alpha)}-
      \kappa\overline{[e^{-\pi i n_jx}\tilde{U}_0^{[9n]}]_{[-18n, 18n]}(x_0+l\alpha)}\|_{\T}\\
< &2L^{-\frac{1}{3}}+e^{-4\epsilon_1 n},
\end{align*}
for $|l|\leq L^{\frac{1}{2}}$. Notice that 
\begin{align*}
[e^{-\pi i n_jx}\tilde{U}_0^{[9n]}]_{[-18n, 18n]}(x)e^{2\pi i n_jx}-\kappa\overline{[e^{-\pi i n_jx}\tilde{U}_0^{[9n]}]_{[-18n, 18n]}(x)}
\end{align*} 
is a polynomial whose essential degree is at most $37n$. Thus by Lemma $\ref{polynomialestimate}$, we would have 
\begin{align*}
\|[e^{-\pi i n_jx}\tilde{U}_0^{[9n]}]_{[-18n, 18n]}(x)e^{\pi i n_jx}-\kappa\overline{[e^{-\pi i n_jx}\tilde{U}_0^{[9n]}]_{[-18n, 18n]}(x)e^{\pi i n_jx}}\|_{\T}<L^{-\frac{1}{4}}+e^{-2\epsilon_1 n}.
\end{align*}
Hence $\|\tilde{U}_0(x)-\kappa \overline{\tilde{U}_0(x)}\|_{\T}<L^{-\frac{1}{4}}+2e^{-2\epsilon_1 n}$. 
But combining with $(9.1)$ we would get $\|\tilde{U}_0(x_0+\frac{L}{4}\alpha)\|<2L^{-\frac{1}{4}}+2e^{-2\epsilon_1 n}$, 
but this contradicts with $\inf_{x\in \T}\|\tilde{U}_0(x)\|>e^{-2\delta n}$ since $\delta<\frac{\epsilon_0}{16}$.$\hfill{} \Box$
\end{proof}

Now for $n>N_0(\lambda,\alpha)$, take $S(x)=\mathrm{Re}\tilde{U}_0(x)$ and $T(x)=\mathrm{Im}\tilde{U}_0(x)$. Let $W_1(x)$ be the matrix with columns $S(x)$ and $T(x)$. Notice that $\det{W}_1(x)$ is well-defined on $\T$ and $\det{W}_1 (x)\neq 0$ on $\T$, hence without loss of generality we could assume $\det{W_1}(x)>0$ on $\T$, otherwise we simply take $W_1(x)$ to be the matrix with columns $S(x)$ and $-T(x)$. Then
\begin{align*}
\|\tilde{A}(x) W_1(x)-W_1(x+\alpha)R_{-\tilde{\theta}}\|_{\T}\leq Ce^{-\frac{\epsilon_1}{45}N}.
\end{align*}
By taking determinant, we get
\begin{align*}
\det{W_1}(x)=\det{W_1}(x+\alpha)+O(e^{-\frac{\epsilon_1}{50}N})\ \ \mathrm{on}\ \T.
\end{align*}
Since $\det{W_1}(x)$ is analytic on $|\mathrm{Im}x|\leq \frac{\epsilon_1}{20\pi}$, by considering the Fourier coefficients we could get
\begin{align*}
\det{W_1}(x)=w_0+O(e^{-\frac{\epsilon_1}{100}N})\ \ \mathrm{on}\ \T,
\end{align*}
where $w_0\geq L^{-5C}$. Thus $\det{W_1}(x)$ is almost a positive constant.

Define $W_2(x)={\det{W_1(x)}}^{-\frac{1}{2}}W_1(x)$. Then $W_2(x)\in C^{\omega}(\T)$ and $\det{W_2(x)}=1$. We have
\begin{align*}
W_2^{-1}(x+\alpha)\tilde{A}(x)W_2(x)=\frac{{\det{W_1(x+\alpha)}}^{\frac{1}{2}}}{{\det{W_1(x)}}^{\frac{1}{2}}}R_{-\tilde{\theta}}+O(e^{-\frac{\epsilon_1}{100}N})\ \ \mathrm{on}\ \T,
\end{align*}
\begin{align*}
{W_2^{-1}(x+\alpha)} \tilde{A}(x) W_2(x)=R_{-\tilde{\theta}} + O(e^{-\frac{\epsilon_1}{200}N})\ \ \mathrm{on}\ \T.
\end{align*}
Now let's prove $\deg{W_2}(x) \leq 36 n$. $\deg{W_2}(x)$ is the same as the degree of its columns. For \linebreak $M: {\R}/{2\Z}\rightarrow {\R^2}$, we say $\deg{M}=k$ if $M$ is homotopic to $\left(\begin{matrix}\cos{k\pi x}\\ \sin{k\pi x}\end{matrix}\right)$.

For some constant $c>0$, we obviously have
\begin{align*}
\int_{\T}\|S(x)\|\ \mathrm{d}x+ \int_{\T}\|T(x)\|\ \mathrm{d}x \geq \int_{\T}\|S(x)+i T(x)\|\ \mathrm{d}x=\int_{\T} \|\tilde{U}_0(x)\|\ \mathrm{d}x \geq c.
\end{align*}
Without loss of generality we could assume $\int_{\T}\|S(x)\|\ \mathrm{d}x>\frac{c}{2}$. Also
\begin{align*}
\tilde{A}(x) S(x)=S(x+\alpha)\cos{2\pi\tilde{\theta}}-T(x+\alpha)\sin{2\pi\tilde{\theta}}+O(e^{-\frac{\epsilon_1}{45}N})\ \ \mathrm{on}\ \T.
\end{align*}
Then since $\|2\tilde{\theta}\|=L^{-1}$,
\begin{align*}
\tilde{A}(x) S(x)=S(x+\alpha)+O(L^{-\frac{1}{2}})\ \ \mathrm{on}\ \T.
\end{align*}
First we prove $\inf_{x\in {\T}} \|S(x)\|\geq e^{-2\epsilon_1 n}$. 
Suppose otherwise. Then there exists $x_0\in \T$, so that $\|S(x_0)\|<e^{-2\epsilon_1 n}$. 
Then $\|\mathrm{Re}\tilde{U}_0(x_0+l\alpha)\|<e^{-\frac{\epsilon_0}{8} n}$ for $|l|<e^{\frac{\epsilon_0}{4C}n}$, where $C$ is the constant that appeared in Theorem $\ref{polycontrol}$.
We have already shown that
\begin{align*}
\|\tilde{U}_0(x)-[e^{-\pi i n_jx}\tilde{U}_0^{[9n]}]_{[-18n, 18n]}e^{\pi i n_jx}\|_{\T}<e^{-4\epsilon_1 n}.
\end{align*}
Thus 
$$\|\mathrm{Re}[e^{-\pi i n_jx}\tilde{U}_0^{[9n]}]_{[-18n, 18n]}(x_0+l\alpha)\|<e^{- \frac{\epsilon_0}{16}n}$$
for $|l|<e^{\frac{\epsilon_0}{4C}n}$.
However $\mathrm{Re}[e^{-\pi i n_jx}\tilde{U}_0^{[9n]}]_{[-18n, 18n]}$ is a polynomial with essential degree at most $36n$. 
Using Lemma $\ref{polynomialestimate}$ we are able to get $\|\mathrm{Re}[e^{-\pi i nx}\tilde{U}_0^{[9n]}]_{[-18n, 18n]}e^{\pi i n_jx}\|_{\T}<e^{-\frac{\epsilon_0}{32}n}$, and thus
$\|\mathrm{Re}\tilde{U}_0(x)\|_{\T}<e^{-\frac{\epsilon_0}{64}n}$ which is a contradiction to $\int_{\T}\|\mathrm{Re}\tilde{U}_0(x)\|\ \mathrm{d}x>\frac{c}{2}$.
At the meantime, we also get $\|S(x)-\mathrm{Re}[e^{-\pi i n_jx}\tilde{U}_0^{[9n]}]_{[-18n, 18n]}(x)e^{\pi i n_jx}\|_{\T} \triangleq \|S(x)-h(x)\|_{\T} \leq e^{-4\epsilon_1 n}$.
The first column of $W_2(x)$ is ${\det{W}_1(x)}^{-\frac{1}{2}}S(x)$. We have
\begin{align*}
       &\|\frac{S(x)}{{\det{W}_1(x)}^{\frac{1}{2}}}-\frac{h(x)}{{w_0}^{\frac{1}{2}}}\|\\
\leq &\frac{1}{|{\det{W}_1(x)}^{\frac{1}{2}}|}\|S(x)-h(x)+(1-\frac{{\det{W}_1(x)}^{\frac{1}{2}}}{{w_0}^{\frac{1}{2}}})h(x)\|\\
\leq & L^{2C}(e^{-4\epsilon_1 n}+L^{8C} e^{-\frac{\epsilon_1}{100}N})\\
\leq & e^{-3\epsilon_1 n}< \|\frac{S(x)}{{\det{W}_1(x)}^{\frac{1}{2}}}\|\ \ \mathrm{on}\ \T.
\end{align*}
Thus by Rouch$\acute{e}$'s theorem $|\deg{W}_2(x)|=|\deg{h}(x)|\leq 19n$.
Notice that
\begin{align*}
|\rho(\alpha, W^{-1}_2\tilde{A}W_2)+\tilde{\theta}|<C e^{-\frac{\epsilon_1}{200}N}.
\end{align*} 
Then, by $\ref{rhoconju}$ for some $|m|\leq 19n$:
\begin{align*}
|\rho(\alpha, \tilde{A})-\frac{m}{2}\alpha+\tilde{\theta}|<C e^{-\frac{\epsilon_1}{200}N}.
\end{align*}

\appendix
\section{\\}
When $\lambda$ belongs to region II, let $\epsilon_2=\ln{\frac{\lambda_2+\sqrt{\lambda_2^2-4\lambda_1\lambda_3}}{\lambda_1+\lambda_3+\sqrt{(\lambda_1+\lambda_3)^2-4\lambda_1\lambda_3}}}>\epsilon_1$. Then $c(x)$ is analytic and nonzero on $|\mathrm{Im}(x)|<\frac{\epsilon_2}{2\pi}$. Furthermore, the winding number of $c(\cdot +i\epsilon)$ is equal to zero when $|\epsilon|< \frac{\epsilon_2}{2\pi}$.
\begin{lemma}\label{c/|c|}
When $\lambda$ belongs to region II, we can find an analytic function $f(x)$ on $|\mathrm{Im}(x)|\leq \frac{\epsilon_1}{2\pi}$ such that $c(x)=|c|(x)e^{f(x+\alpha)-f(x)}$ and $\tilde{c}(x)=|c|(x)e^{-f(x+\alpha)+f(x)}$.
\end{lemma}
\begin{proof}
Since the winding numbers of $c(x)$ and $\tilde{c}(x)$ are $0$ on $|\mathrm{Im}(x)|\leq \frac{\epsilon_1}{2\pi}$, 
there exist analytic functions $g_1(x)$ and $g_2(x)$ on $|\mathrm{Im}(x)|\leq\frac{\epsilon_1}{2\pi}$, such that $c(x)=e^{g_1(x)}$ and $\tilde{c}(x)=e^{g_2(x)}$. 
Notice that
\begin{align*}
\int_{\T} \ln{|c(x)|}\ \mathrm{d}x=\int_{\T} \ln{|\tilde{c}(x)}|\ \mathrm{d}x\\
\int_{\T} \arg{c(x)}\ \mathrm{d}x=\int_{\T} \arg{\tilde{c}(x)}\ \mathrm{d}x,
\end{align*}
so there exists an analytic function $f(x)$ such that $2f(x+\alpha)-2f(x)=g_1(x)-g_2(x)$. Then $c(x)=|c|(x)e^{f(x+\alpha)-f(x)}$.
\end{proof} $\hfill{} \Box$

\begin{lemma}\label{conjugate}
When $\lambda$ belongs to region II, there exists an analytic matrix $Q_{\lambda}(x)$ defined on $|\mathrm{Im}(x)|\leq \frac{\epsilon_1}{2\pi}$ such that
\begin{align*}
Q_{\lambda}^{-1}(x+\alpha)\tilde{A}_{\lambda, E}(x)Q_{\lambda}(x)=A_{\lambda, E}(x).
\end{align*}
\end{lemma}

\begin{proof}
\begin{align*}
\tilde{A}_{\lambda, E}(x)=&
\frac{1}{\sqrt{|c|(x)|c|(x-\alpha)}}
\left(
\begin{matrix}
1       &0\\
0       &\sqrt{\frac{\tilde{c}(x)}{c(x)}}
\end{matrix}
\right)
\left(
\begin{matrix}
E-v(x)        &-\tilde{c}(x-\alpha)\\
c(x)  &0
\end{matrix}
\right)
\left(
\begin{matrix}
1        &0\\
0        &\sqrt{\frac{c(x-\alpha)}{\tilde{c}(x-\alpha)}}
\end{matrix}
\right)\\
=&
\frac{c(x)}{\sqrt{|c|(x)|c|(x-\alpha)}}
\left(
\begin{matrix}
1       &0\\
0       &\sqrt{\frac{\tilde{c}(x)}{c(x)}}
\end{matrix}
\right)
A(x)
\left(
\begin{matrix}
1        &0\\
0        &\sqrt{\frac{c(x-\alpha)}{\tilde{c}(x-\alpha)}}
\end{matrix}
\right)\\
=&
e^{f(x+\alpha)}\sqrt{|c|(x)}
\left(
\begin{matrix}
1       &0\\
0       &\sqrt{\frac{\tilde{c}(x)}{c(x)}}
\end{matrix}
\right)
A(x)
\left\lbrace e^{f(x)}\sqrt{|c|(x-\alpha)}
\left(
\begin{matrix}
1        &0\\
0        &\sqrt{\frac{\tilde{c}(x-\alpha)}{c(x-\alpha)}}
\end{matrix}
\right)
\right\rbrace^{-1}\\
=&Q_{\lambda}(x+\alpha)A_{\lambda, E}(x)Q_{\lambda}^{-1}(x).
\end{align*}
\end{proof} $\hfill{} \Box$

\begin{lemma}\label{LEregion2}
If $\alpha$ is irrational, $\lambda$ belongs to region II, $E\in\Sigma (\lambda)$, then $L(\alpha, A_{\lambda, E}(\cdot+i\epsilon))=L(\alpha, \tilde{A}_{\lambda, E}(\cdot+i\epsilon))=0$ for $|\epsilon|\leq \frac{\epsilon_1}{2\pi}$.
\end{lemma}
\begin{proof}
$L(A(\cdot+i \epsilon))=L(D(\cdot +i\epsilon))-\int \ln{|c(x+i\epsilon)|}\mathrm{d}x$
\begin{align*}
D(x+i\epsilon)
&=\left(
\begin{matrix}
E-e^{2\pi i (x+i\epsilon)}-e^{-2\pi i (x+i\epsilon)}\ \                                                &-\lambda_1 e^{2\pi i (x-\frac{\alpha}{2}+i\epsilon)}-\lambda_2-\lambda_3 e^{-2\pi i (x-\frac{\alpha}{2}+i\epsilon)}\\
\lambda_1 e^{-2\pi i(x+\frac{\alpha}{2}+i\epsilon)}+\lambda_2+\lambda_3 e^{2\pi i (x+\frac{\alpha}{2}+i\epsilon)}\ \  &0
\end{matrix}
\right)\\
&=e^{2\pi \epsilon}
\left(
\begin{matrix}
-e^{2\pi i x}+o(1)\ \ & -\lambda_3 e^{-2\pi i (x-\frac{\alpha}{2})}+o(1)\\
\lambda_1 e^{-2\pi i (x+\frac{\alpha}{2})}+o(1)\ \ &0
\end{matrix}
\right).
\end{align*}
Thus the asymptotic behaviour of $L(D(\cdot+i\epsilon))$ is:
\begin{align*}
&L(D(\cdot+i\epsilon))=\ln{|\frac{1+\sqrt{1-4\lambda_1\lambda_3}}{2}|}+2\pi \epsilon\ \ \mathrm{when}\ \epsilon\rightarrow \infty,\\
&L(D(\cdot+i\epsilon))=\ln{|\frac{1+\sqrt{1-4\lambda_1\lambda_3}}{2}|}-2\pi \epsilon\ \ \mathrm{when}\ \epsilon\rightarrow -\infty.
\end{align*}
Then it suffices to calculate $\int \ln{|c(x+i\epsilon)|}\mathrm{d}x$ in region II. We have
\begin{align*}
&\int \ln{|c(x+i\epsilon)|}\mathrm{d} x\\
=&\ln{\lambda_3}-2\pi \epsilon+\int \ln{|e^{2\pi i x}-y_{1,\epsilon}|}+\int \ln{|e^{2\pi i x}-y_{2,\epsilon}|}.
\end{align*}
where $y_{1,\epsilon}=\frac{-\lambda_2+\sqrt{\lambda_2^2-4\lambda_1\lambda_3}}{2\lambda_3}e^{2\pi \epsilon}$ and $y_{2,\epsilon}=\frac{-\lambda_2-\sqrt{\lambda_2^2-4\lambda_1\lambda_3}}{2\lambda_3}e^{2\pi \epsilon}$.
\begin{align*}
\int \ln{|c(x+i\epsilon)|}\mathrm{d}x=
\left\lbrace
\begin{matrix}
2\pi \epsilon+\ln{\lambda_1}\ \ \ \ \ \ \ \ \ \ \ \ \ \ &\epsilon>\frac{1}{2\pi}\ln{\frac{\lambda_2+\sqrt{\lambda_2^2-4\lambda_1\lambda_3}}{2\lambda_1}},\\
\\
\ln{\frac{\lambda_2+\sqrt{\lambda_2^2-4\lambda_1\lambda_3}}{2}}\ \ \ \ &\frac{1}{2\pi}\ln{\frac{\lambda_2-\sqrt{\lambda_2^2-4\lambda_1\lambda_3}}{2\lambda_1}}\leq \epsilon\leq \frac{1}{2\pi}\ln{\frac{\lambda_2+\sqrt{\lambda_2^2-4\lambda_1\lambda_3}}{2\lambda_1}},\\
\\
-2\pi \epsilon+\ln{\lambda_3}\ \ \ \ \ \ \ \ \ \ \ \ \ &\epsilon<\frac{1}{2\pi}\ln{\frac{\lambda_2-\sqrt{\lambda_2^2-4\lambda_1\lambda_3}}{2\lambda_1}}.
\end{matrix}
\right.
\end{align*}
Thus $L(A(\cdot +i\epsilon))=0$ when $|\epsilon|\leq \frac{1}{2\pi}\ln{\frac{\lambda_2+\sqrt{\lambda_2^2-4\lambda_1\lambda_3}}{\max{(1,\lambda_1+\lambda_3)}+\sqrt{\max{(1,\lambda_1+\lambda_3)}^2-4\lambda_1\lambda_3}}}=\frac{\epsilon_1}{2\pi}$.

Since $\tilde{A}_{\lambda, E}(x+i\epsilon)=Q_{\lambda}(x+\alpha+i\epsilon)A_{\lambda, E}(x+i\epsilon)Q_{\lambda}^{-1}(x+i\epsilon)$, the statement about $\tilde{A}_{\lambda, E}$ is also true.
\end{proof} $\hfill{} \Box$

\section*{Acknowledgement}
I am deeply grateful to Svetlana Jitomirskaya for suggesting this problem and for many valuable discussions. 
This research was partially supported by the NSF DMS–1401204.

\bibliographystyle{amsplain}

\end{document}